\documentclass[10pt]{article}%

\usepackage{rotating} % needed for \begin{sidewaysfigure}\begin{sidewaystable}

\usepackage{mathrsfs}

\DeclareMathAlphabet{\mathpzc}{OT1}{pzc}{m}{it}

\usepackage{color}
\usepackage{amsmath}
\usepackage{graphicx}
\usepackage{amsfonts}
\usepackage{amssymb}%
\usepackage{latexsym}
\usepackage{psfrag}
\usepackage{subfigure}
\usepackage{accents}
\usepackage{bm}

\newtheorem{theorem}{Theorem}[section]
\newtheorem{corollary}[theorem]{Corollary}
\newtheorem{definition}[theorem]{Definition}
\newenvironment{proof}[1][Proof]{\noindent \emph{#1.} }
{\hfill \ \rule{0.5em}{0.5em}}
\newtheorem{lemma}[theorem]{Lemma}
\newtheorem{proposition}[theorem]{Proposition}

\newtheorem{example}[theorem]{Example}
\newtheorem{remark}[theorem]{Remark}

\numberwithin{equation}{section}

\usepackage[a4paper]{geometry}
\newcommand{\noi}{\noindent}

\newcommand{\R}{\mathbb{R}}

\newcommand{\cS}{{\cal S}}

\newcommand{\cD}{{\cal D}}

\newcommand{\bx}{x}

\newcommand{\bn}{n}%\mathbf{n}}

\newcommand{\by}{y}%\mathbf{y}}

\newcommand{\supp}{\mathrm{supp}}

\newcommand{\re}{{\rm e}}
\newcommand{\ri}{{\rm i}}
\newcommand{\rd}{{\rm d}}

\newcommand{\beq}{\begin{equation}}
\newcommand{\eeq}{\end{equation}}
\newcommand{\beqs}{\begin{equation*}}
\newcommand{\eeqs}{\end{equation*}}
\newcommand{\bit}{\begin{itemize}}
\newcommand{\eit}{\end{itemize}}
\newcommand{\ben}{\begin{enumerate}}
\newcommand{\een}{\end{enumerate}}
\newcommand{\bal}{\begin{align}}
\newcommand{\eal}{\end{align}}
\newcommand{\bals}{\begin{align*}}
\newcommand{\eals}{\end{align*}}
\newcommand{\bse}{\begin{subequations}}
\newcommand{\ese}{\end{subequations}}
\newcommand{\bpr}{\begin{proposition}}
\newcommand{\epr}{\end{proposition}}
\newcommand{\bre}{\begin{remark}}
\newcommand{\ere}{\end{remark}}
\newcommand{\bpf}{\begin{proof}}
\newcommand{\epf}{\end{proof}}
\newcommand{\ble}{\begin{lemma}}
\newcommand{\ele}{\end{lemma}}
\newcommand{\bco}{\begin{corollary}}
\newcommand{\eco}{\end{corollary}}
\newcommand{\bex}{\begin{example}}
\newcommand{\eex}{\end{example}}
\newcommand{\bth}{\begin{theorem}}
\newcommand{\enth}{\end{theorem}}

\newcommand{\Rea}{\mathbb{R}}

\newcommand{\Oi}{{\Omega}}%\Omega_-}}

\newcommand{\Oe}{{\Omega_+}}%{\Omega_+}}

\newcommand{\eps}{\varepsilon}

\newcommand{\pdiff}[2]{\frac{\partial #1}{\partial #2}}

\newcommand{\dnpu}{\partial_n^+ u}

\newcommand{\gpu}{\gamma^+ u}

\newcommand{\nT}{\nabla_{\bound}}

\newcommand{\half}{\frac{1}{2}}

\newcommand{\bound}{{\partial \Oi}}

\newcommand{\LtG}{{L^2(\bound)}}

\newcommand{\LtGt}{{\LtG\rightarrow \LtG}}
\newcommand{\LtGtH}{{\LtG\rightarrow \HoG}}

\newcommand{\HhG}{{H^{1/2}(\bound)}}
\newcommand{\HmhG}{{H^{-1/2}(\bound)}}

\newcommand{\HoG}{H^1(\bound)}
\newcommand{\HohG}{H^1_k(\bound)}

\newcommand{\tendi}{\rightarrow \infty}
\newcommand{\tendo}{\rightarrow 0}

\newcommand{\opA}{A'_{k,\eta}}
\newcommand{\opABW}{A_{k,\eta}}
\newcommand{\opAinv}{(A'_{k,\eta})^{-1}}
\newcommand{\opABWinv}{A^{-1}_{k,\eta}}

\newcommand{\normAinv}{\|\opAinv\|}

\def\XXint#1#2#3{{\setbox0=\hbox{$#1{#2#3}{\int}$}
     \vcenter{\hbox{$#2#3$}}\kern-.5\wd0}}

\usepackage{hyperref}
\usepackage{tikz}
\definecolor{myblue}{rgb}{0,0,0.6}
\hypersetup{colorlinks=true,
linkcolor=myblue,citecolor=myblue,filecolor=myblue,urlcolor=myblue}
\newcommand*{\N}[1]{\left\|#1\right\|}

\allowdisplaybreaks[4]
\newcommand{\tand}{\text{ and }}

\newcommand{\vertiii}[1]{{\left\vert\kern-0.25ex\left\vert\kern-0.25ex\left\vert #1 
    \right\vert\kern-0.25ex\right\vert\kern-0.25ex\right\vert}}

\newcommand{\newx}{x}
\newcommand{\newxi}{\xi}

\newcommand{\Gammaext}{\widetilde{\Gamma}}

\DeclareMathOperator*{\loc}{loc}
\DeclareMathOperator*{\comp}{comp}
\DeclareMathOperator*{\Ell}{ell}

\usepackage{chngcntr}

\graphicspath{{Figures/}}

\renewcommand{\Im}{\mathop{\rm Im}\nolimits}

\newcommand{\bl}{\begin{flushleft}}
\newcommand{\el}{\end{flushleft}}
\newcommand{\ert}{\end{flushright}}
\newcommand{\bc}{\begin{center}}
\newcommand{\ec}{\end{center}}
\newcommand{\numList}{\begin{enumerate}}
\newcommand{\enumList}{\end{enumerate}}

\newcommand{\e}{\epsilon}

\newcommand{\la}{\langle}
\newcommand{\ra}{\rangle}

\newcommand{\mc}[1]{\mathcal{#1}}
\newcommand{\pO}{\bound}

\newcommand{\Dl}{D_k}

\newcommand{\Cc}{C_c^\infty}
\renewcommand{\O}[1]{\mathcal{O}_{#1}}

\newcommand{\Ph}[2]{\Psi^{#1}_{#2}}

\newcommand{\Qs}{Q_S}
\newcommand{\Qa}{Q_D}

\newcommand{\Qap}{Q_{D'}}
\newcommand{\Sk}{S_k}
\newcommand{\Dkreg}{C^{2,\alpha}}

\begin{document}

\title{Wavenumber-explicit regularity estimates on the acoustic single- and double-layer operators}
\author{Jeffrey Galkowski\footnotemark[1]\,,\, Euan A.~Spence\footnotemark[2]}

\date{\today}

\renewcommand{\thefootnote}{\fnsymbol{footnote}}

\footnotetext[1]{Department of Mathematics, Stanford University, Building 380, Stanford, California 94305, USA \tt jeffrey.galkowski@stanford.edu}
\footnotetext[2]{Department of Mathematical Sciences, University of Bath, Bath, BA2 7AY, UK, \tt E.A.Spence@bath.ac.uk }

\maketitle

\begin{abstract}
We prove new, sharp, wavenumber-explicit bounds on the norms of the Helmholtz single- and double-layer boundary-integral operators as mappings from $\LtG\rightarrow\HoG$ (where $\bound$ is the boundary of the obstacle).
The new bounds are obtained using estimates on the restriction to the boundary of quasimodes of the Laplacian, building on recent work by the first author and collaborators.

Our main motivation for considering these operators is that they appear in the standard second-kind boundary-integral formulations, posed in $\LtG$, of the exterior Dirichlet problem for the Helmholtz equation. Our new wavenumber-explicit $\LtG\rightarrow\HoG$ bounds can then be used in a wavenumber-explicit version of the classic compact-perturbation analysis of Galerkin discretisations of these second-kind equations; this is done in the companion paper \emph{[Galkowski, M\"uller, Spence, arXiv 1608.01035]}.

\

\noi \textbf{Keywords:} Helmholtz equation, layer-potential operators, high frequency, semiclassical, boundary integral equation. 

\

\noi\textbf{AMS Subject Classifications:} 31B10, 31B25, 35J05, 35J25, 65R20

\end{abstract}

\renewcommand{\thefootnote}{\arabic{footnote}}

\section{Introduction}\label{sec:intro}

\subsection{Statement of the main results}

Let $\Phi_k(\bx,\by)$ be the fundamental solution of the Helmholtz equation ($\Delta u +k^2u=0$) given by 
\beq\label{eq:fund}
\Phi_k(\bx,\by):=\displaystyle\frac{\ri}{4}H_0^{(1)}\big(k|\bx-\by|\big), \,\,d=2,\quad\quad \Phi_k(\bx,\by) := \frac{\re^{\ri k |\bx-\by|}}{4\pi |\bx-\by|}, \,\,d=3,
\eeq
where $d$ is the spatial dimension.
Let $\Omega$ be a bounded Lipschitz open set such that the open complement $\Oe:= \Rea^d\setminus \overline{\Omega}$ is connected (so that the scattering problem with obstacle $\Omega$ is well-defined). 
Recall that, for almost every $x\in \partial \Omega$, there exists a unique outward-pointing unit normal vector, which we denote by $n(x)$.
For $\phi\in\LtG$ and $x\in\bound$,
the single- and double-layer potential operators are defined by 
\beq
S_k \phi(\bx) := \int_\bound \Phi_k(\bx,\by) \phi(\by)\,\rd s(\by), \quad
D_k \phi(\bx) := \int_\bound \frac{\partial \Phi_k(\bx,\by)}{\partial n(\by)}  \phi(\by)\,\rd s(\by),
 \label{eq:SD}
\eeq
and the adjoint-double-layer operator is defined by 
\beq \label{eq:D}
D'_k \phi(\bx) := \int_\bound \frac{\partial \Phi_k(\bx,\by)}{\partial n(\by)}  \phi(\by)\,\rd s(\by)
\eeq
(recall that $D'_k$ is the adjoint of $D_k$ with respect to the real-valued $L^2(\bound)$ inner product; see, e.g., \cite[Page 120]{ChGrLaSp:12}).

Before stating our main results, we need to make the following definitions.

\begin{definition}[Smooth hypersurface]\label{def:sh}
We say that $\Gamma\subset \R^d$ is a {\em smooth hypersurface} if there exists $\Gammaext$ a compact embedded 
 smooth $d-1$ dimensional submanifold of $\R^d$, possibly with boundary,
 such that $\Gamma$ is an open subset of $\Gammaext$, with $\Gamma$ strictly away from $\partial \Gammaext$,
 and the boundary of $\Gamma$ can be written as a disjoint union
\beqs
\partial \Gamma=\left(\bigcup_{\ell=1}^n Y_\ell\right)\cup \Sigma,
\eeqs
where each $Y_\ell$ is an open, relatively compact, smooth embedded manifold of dimension $d-2$ in $\Gammaext$, $\Gamma$ lies locally on one side of $Y_\ell$, and  $\Sigma$ is closed set with $d-2$ measure $0$ and $\Sigma \subset \overline{\bigcup_{l=1}^nY_l}$. We then refer to the manifold $\Gammaext$ as an extension of $\Gamma$. 
\end{definition}

\noi For example, when $d=3$, the interior of a 2-d polygon is a smooth hypersurface, with $Y_i$ the edges and $\Sigma$ the set of corner points.

\begin{definition}[Curved]\label{def:curved}
We say a smooth hypersurface is \emph{curved} if there is a choice of normal so that the second fundamental form of the hypersurface is everywhere positive definite.
\end{definition}

\noi Recall that the principal curvatures are the eigenvalues of the matrix of the second fundamental form in an orthonormal basis of the tangent space, and thus ``curved" is equivalent to the principal curvatures being everywhere strictly positive (or everywhere strictly negative, depending on the choice of the normal).

\begin{definition}[Piecewise smooth]\label{def:psh}
We say that a hypersurface $\Gamma$ is \emph{piecewise smooth} if $\Gamma=\cup_{i=1}^N \overline{\Gamma}_i$ where $\Gamma_i$ are smooth hypersurfaces 
and $\Gamma_i\cap \Gamma_j=\emptyset.$
\end{definition}

\begin{definition}[Piecewise curved]\label{def:pc}
We say that a piecewise smooth hypersurface $\Gamma$ is \emph{piecewise curved} if $\Gamma$ is as in Definition \ref{def:psh} and each $\Gamma_j$ is curved.
\end{definition}

The main results of this paper are contained in the following theorem. We use the notation that $a\lesssim b$ if there exists a $C>0$, independent of $k$, such that $a\leq C b$.

\begin{theorem}[Bounds on $\|S_k\|_{\LtG\rightarrow\HoG}$, $\|D_k\|_{\LtG\rightarrow\HoG}$, $\|D'_k\|_{\LtG\rightarrow\HoG}$]
\label{thm:L2H1}

\

\noi Let $\Omega$ be a bounded Lipschitz open set such that the open complement $\Oe:= \Rea^d\setminus \overline{\Omega}$ is connected.

\noi (a) 
If $\bound$ is a piecewise smooth hypersurface (in the sense of Definition \ref{def:psh}), then, given $k_0>1$, 
\begin{equation}
\label{eqn:optimalFlatSl}
\N{S_k}_{\LtG\rightarrow\HoG} \lesssim k^{1/2}\log k, 
\end{equation}
for all $k\geq k_0$.
Moreover, if $\bound$ is piecewise curved (in the sense of Definition \ref{def:pc}),  then, given $k_0>1$, the following stronger estimate holds for all $k\geq k_0$
\begin{equation}\label{eq:k13log}
\N{S_k}_{\LtG\rightarrow\HoG} \lesssim k^{1/3}\log k.
\end{equation}

\noi (b) If $\bound$ is a piecewise smooth, $\Dkreg$ hypersurface, for some $\alpha>0$, then, given $k_0>1$, 
\begin{equation*}
\N{D_k}_{\LtG\rightarrow\HoG} + \N{D'_k}_{\LtG\rightarrow\HoG} \lesssim k^{5/4}\log k
\end{equation*}
for all $k\geq k_0$. 
Moreover, if $\bound$ is piecewise curved, 
then, given $k_0>1$, the following stronger estimates hold for all $k\geq k_0$
\begin{equation*}
\N{D_k}_{\LtG\rightarrow\HoG} + \N{D'_k}_{\LtG\rightarrow\HoG} \lesssim k^{7/6}\log k.
\end{equation*}

\noi (c) If $\Oi$ is convex and $\bound$ is $C^\infty$ and curved (in the sense of Definition \ref{def:curved}) then, given $k_0>0$,
\begin{align}
&\qquad\qquad\N{S_k}_{\LtG\rightarrow\HoG} \lesssim k^{1/3}, \label{eq:k13} \\ \nonumber
&\N{D_k}_{\LtG\rightarrow\HoG}+\N{D'_k}_{\LtG\rightarrow\HoG} \lesssim k
\end{align}
for all $k\geq k_0$. 
\end{theorem}

Note that the requirement in Part (b) of Theorem \ref{thm:L2H1} that $\bound$ is $C^{2,\alpha}$ arises since this is the regularity required of $\bound$ for $D_k$ and $D'_k$ to map $\LtG$ to $\HoG$; see \cite[Theorem 4.2]{Ki:89},  \cite[Theorem 3.6]{CoKr:98}.

\bre[Sharpness of the bounds in Theorem \ref{thm:L2H1}]
In Section \ref{app:sharpness} we show that, modulo the factor $\log k$, all of the bounds in Theorem \ref{thm:L2H1} are sharp (i.e.~the powers of $k$ in the bounds are optimal). 
The sharpness (modulo the factor $\log k$) of the $\LtG\rightarrow\LtG$ bounds in Theorem \ref{thm:L2H12} was proved in 
\cite[\S A.2-A.3]{HaTa:15}. Earlier work in \cite[\S4]{ChGrLaLi:09} proved the sharpness of some of the $\LtG\rightarrow\LtG$ bounds in 2-d; we highlight that Section \ref{app:sharpness} and \cite[\S A.2-A.3]{HaTa:15} contain the appropriate generalisations to multidimensions of some of the arguments of \cite[\S4]{ChGrLaLi:09} (in particular \cite[Theorems 4.2 and 4.4]{ChGrLaLi:09}).
\ere

\bre[Comparison to previous results]\label{rem:GrLoMeSp}
The only previously-existing bounds on the $\LtG\rightarrow\HoG$-norms of $S_k$, $D_k$, and $D_k'$ are the following:
\beq\label{eq:Sb}
\N{S_k}_{\LtGtH} \lesssim k^{(d-1)/2}
\eeq
when $\bound$ is Lipschitz \cite[Theorem 1.6 (i)]{GrLoMeSp:15}, and 
\beq\label{eq:Db}
\N{D_k}_{\LtGtH} + \N{D'_k}_{\LtGtH} \lesssim k^{(d+1)/2}
\eeq
when $\bound$ is $C^{2,\alpha}$ \cite[Theorem 1.6 (ii)]{GrLoMeSp:15}. 

We see that \eqref{eq:Sb} is a factor of $\log k$ sharper than the bound \eqref{eqn:optimalFlatSl} when $d=2$, but otherwise all the bounds in Theorem \ref{thm:L2H1} are sharper than \eqref{eq:Sb} and \eqref{eq:Db}.
\ere

\bre[Bounds for general dimension and $k\in \Rea$]
We have restricted attention to $2$- and $3$-dimensions because these are the most practically-interesting ones. From a semiclassical point of view, it is natural work in $d\geq 1$, and the results of Theorem \ref{thm:L2H1} apply for any $d\geq 1$ (although when $d=1$ it is straightforward to get sharper bounds; see \cite[\S1]{GaSm:15}).
We have also restricted attention to the case when $k$ is positive and bounded away from $0$. Nevertheless, the methods used to prove the bounds in Theorem \ref{thm:L2H1} show that if one replaces $\log k$ by $\log \la k \ra$ (where $\la \cdot \ra = (2+ |\cdot|^2)^{1/2}$) and includes an extra factor of $\log \la k^{-1} \ra$ when $d=2$, then the resulting bounds hold for all $k\in \Rea$ 
\ere

As explained in \S\ref{sec:motivation} below, the motivation for proving the $\LtG\rightarrow \HoG$ bounds of Theorem \ref{thm:L2H1} comes from interest in second-kind Helmholtz boundary integral equations (BIEs) posed in $\LtG$. However, there is also a large interest in both first- and second-kind Helmholtz BIEs posed in the trace spaces $\HmhG$ and $\HhG$ (see, e.g., \cite[\S3.9]{SaSc:11}, \cite[\S7.6]{St:08}). The $k$-explicit theory of Helmholtz BIEs in  
the trace spaces is much less developed than the theory in $\LtG$, so we therefore highlight that the 
$\LtG\rightarrow \HoG$ bounds in Theorem \ref{thm:L2H1} can be converted to $H^{s-1/2}(\bound) \rightarrow H^{s+1/2}(\bound)$ bounds for $|s|\leq 1/2$.

\begin{corollary}[Bounds from $H^{s-1/2}(\bound) \rightarrow H^{s+1/2}(\bound)$ for $|s|\leq 1/2$]\label{cor:trace}
Theorem \ref{thm:L2H1} is valid with all the norms from $\LtG\rightarrow \HoG$ replaced by norms from $H^{s-1/2}(\bound) \rightarrow H^{s+1/2}(\bound)$ for $|s|\leq 1/2$.
\end{corollary}

\bre[The idea behind Theorem \ref{thm:L2H1}]
The bounds of Theorem \ref{thm:L2H1} are proved using estimates on the restriction of quasimodes 
of the Laplacian to hypersurfaces from \cite{Tat}, \cite{BGT}, \cite{T}, \cite{HTacy}, \cite{christianson2014exterior}, and \cite{T14} (and recapped in \S\ref{sec:rest} below). 
The reason why these restriction estimates can be used to prove bounds on boundary-integral operators is explained in \S\ref{sec:outline} below; this idea was first introduced in \cite{GaSm:15}, \cite[Appendix A]{HaTa:15} and \cite{Ga:15}, where $\LtG\rightarrow\LtG$ bounds were proved on $S_k, D_k, D'_k$.
\ere

\subsection{Motivation for proving Theorem \ref{thm:L2H1}}\label{sec:motivation}

Our motivation for 
proving Theorem \ref{thm:L2H1} has four parts. 
\ben
\item The integral operators $S_k, D_k,$ and $D'_k$ appear in the standard second-kind BIE formulations of the exterior Dirichlet problem for the Helmholtz equation. 
\item The standard analysis of the Galerkin method applied to these second-kind BIEs is based on the fact that, when $\bound$ is $C^1$, the operators $S_k, D_k,$ and $D'_k$ are all compact, and thus $\opA$ and $\opABW$ are compact perturbations of $\half I$. 
\item To perform a $k$-explicit analysis of the Galerkin method applied to $\opA$ or $\opABW$ via these compact-perturbation arguments, we need to have $k$-explicit information about the smoothing properties of $S_k$, $D_k$, and $D'_k$.
\item 
When $\bound$ is $C^{2,\alpha}$, the operators $S_k$, $D_k$, and $D'_k$ all map $\LtG$ to $\HoG$, and $k$-explicit bounds on these norms therefore give the required $k$-explicit smoothing information. 
\een

Regarding Point 1: if $u$ is the solution of the exterior Dirichlet problem for the Helmholtz equation 
\beqs
\Delta u(\bx) + k^2 u(\bx)= 0, \quad \bx \in \Oe,
\eeqs
satisfying the Sommerfeld radiation condition
\beqs
\pdiff{u}{r}(\bx) - \ri k u(\bx) = 
o\bigg(\frac{1}{r^{(d-1)/2}}\bigg)
\eeqs
as $r := |\bx| \tendi$, uniformly in $\bx/r$, then Green's integral representation theorem implies that 
\beq\label{eq:Green} 
u(\bx) = -\int_\bound \Phi_k(\bx,\by) \dnpu (\by)\rd s(\by) + \int_\bound  \frac{\partial \Phi_k(\bx,\by)}{\partial n(\bx)}  \gpu(\by)\,\rd s(\by), \quad \bx \in \Oe,
\eeq
where $\dnpu$ is the (unknown) Neumann trace on $\bound$ and $\gpu$ is the (known) Dirichlet trace. Taking the Dirichlet and Neumann traces of \eqref{eq:Green}, 
using the jump relations for the single- and double-layer potentials (see, e.g. \cite[Equations 2.41-2.43]{ChGrLaSp:12}),
and then taking a linear combination of the resulting equations, we obtain the so-called ``direct" BIE
\beq\label{eq:direct}
\opA \dnpu = f_{k,\eta}
\eeq
where 
\beq\label{eq:scpo}
\opA := \half I + D'_k -\ri\eta S_k, 
\eeq
$\eta\in \Rea\setminus \{0\}$, and $f_{k,\eta}$ is given in terms of the known Dirichlet trace; see, e.g., \cite[Equation 2.68]{ChGrLaSp:12} (the exact form of $f_{k,\eta}$ is not important for us here).
Alternatively, one can pose the ansatz
\beq\label{eq:ansatz}
u(\bx)=\int_\bound \frac{\partial \Phi_k(\bx,\by)}{\partial n(\by)}  \phi(\by)\,\rd s(\by) - \ri \eta \int_\bound \Phi_k(\bx,\by)\phi(\by)\,\rd s(\by), 
\eeq
for $\bx\in \Oe$, $\phi\in \LtG$, and $\eta\in \Rea\setminus \{0\}$. Taking the Dirichlet trace of \eqref{eq:ansatz}, we obtain the so-called ``indirect" BIE
\beq\label{eq:indirect}
\opABW \phi = \gpu,
\eeq
where 
\beq\label{eq:scpo2}
\opABW := \half I + D_k -\ri\eta S_k.
\eeq
The motivation for considering these ``combined BIEs" (i.e.~BIEs involving a linear combination of $S_k, D_k$, and $D'_k$) is that, when $\eta\in \Rea\setminus\{0\}$, the operators $\opA$ and $\opABW$ are bounded, invertible operators on $\LtG$ for all $k>0$ (see, e.g., \cite[Theorem 2.27]{ChGrLaSp:12}). In contrast,
the integral operators $S_k$, $(\half I + D'_k)$, and $(\half I + D_k)$ are \emph{not} invertible for all $k>0$ (see, e.g., \cite[\S2.5]{ChGrLaSp:12}).

Regarding Point 2: $S_k$ is compact when $\bound$ is Lipschitz (since $S_k:\LtG\rightarrow \HoG$ in this case \cite[Theorem 1.6]{Ve:84}), and $D_k$ and $D'_k$ are compact when $\bound$ is $C^1$ \cite[Theorem 1.2(c)]{FaJo:78}. 

Regarding Points 3 and 4: \cite{GrLoMeSp:15} performed a $k$-explicit version of the classic compact-perturbation argument appearing in, e.g., \cite[Chapter 3]{At:97}. The two $k$-explicit ingredients were the $\LtG\rightarrow \HoG$ bounds on $S_k, D_k,$ and $D'_k$ discussed in Remark \ref{rem:GrLoMeSp} (and proved in \cite[Theorem 1.6]{GrLoMeSp:15}) and the sharp  $\LtG\rightarrow \LtG$ bound on $\opAinv$ and $\opABWinv$ when $\Omega$ is star-shaped with respect to a ball from \cite[Theorem 4.3]{ChMo:08}. The paper \cite{GaMuSp:18} 
shows how the results of \cite{GrLoMeSp:15} are improved by using the new, sharp $\LtG\rightarrow\HoG$ bounds on $S_k, D_k,$ and $D'_k$ from Theorem \ref{thm:L2H1}, along with the sharp $\LtG\rightarrow \LtG$ bounds on  $\opAinv$ and $\opABWinv$ for nontrapping $\Omega$ from \cite[Theorem 1.13]{BaSpWu:16}.

\subsection{Discussion of the results of Theorem \ref{thm:L2H1} in the context of using semiclassical analysis in the numerical analysis of the Helmholtz equation.}

In the last 10 years, there has been growing interest in using results about the $k$-explicit analysis of the Helmholtz equation 
from semiclassical analysis 
to design and analyse numerical methods for the Helmholtz equation\footnote{A closely-related activity is the design and analysis of numerical methods for the Helmholtz equation based on proving \emph{new} results about the $k\tendi$ asymptotics of Helmholtz solutions for polygonal obstacles; see \cite{ChLa:07}, \cite{HeLaMe:13}, \cite{HeLaCh:14}, \cite{ChHeLaTw:15}, and \cite{He:15}.}. The activity has occurred in, broadly speaking, four different directions:

\ben
\item The use of the results of Melrose and Taylor \cite{MeTa:85} -- on the rigorous $k\tendi$ asymptotics of the solution of the Helmholtz equation in the exterior of a smooth convex obstacle with strictly positive curvature -- to design and analyse $k$-dependent approximation spaces for integral-equation formulations \cite{DoGrSm:07}, \cite{GaHa:11}, \cite{AsHu:14}, \cite{EcOz:17}, \cite{EcHa:16}, \cite{Ec:18}. 
\item The use of the results of Melrose and Taylor \cite{MeTa:85}, along with the work of Ikawa \cite{Ik:88} on scattering from several convex obstacles, to analyse algorithms for multiple scattering problems \cite{EcRe:09}, \cite{AnBoEcRe:10}.
\item The use of bounds on the Helmholtz solution operator (also known as \emph{resolvent estimates}) due to Vainberg \cite{Va:75} (using the propagation of singularities results of Melrose and Sj\"ostrand \cite{MeSj:82}) and Morawetz \cite{Mo:75} to prove bounds on both $\normAinv_{\LtGt}$ and the inf-sup constant of the domain-based variational formulation \cite{ChMo:08}, \cite{Sp:14}, \cite{BaSpWu:16}, \cite{ChSpGiSm:17}, and also to analyse preconditioning strategies \cite{GaGrSp:15}.
\item The use of identities originally due to Morawetz \cite{Mo:75} to prove coercivity of $\opA$ \cite{SpKaSm:15} and to introduce new coercive formulations of Helmholtz problems \cite{SpChGrSm:11}, \cite{MoSp:14}.
\een
This paper concerns a fifth direction, namely proving sharp $k$-explicit bounds on $S_k, D_k$ and $D'_k$ 
using estimates on the restriction of quasimodes 
of the Laplacian to hypersurfaces from \cite{Tat}, \cite{BGT}, \cite{T}, \cite{HTacy}, \cite{christianson2014exterior}, and \cite{T14} (and recapped in \S\ref{sec:rest} below). 
This direction was initiated in \cite{GaSm:15}, \cite[Appendix A]{HaTa:15}, and \cite{Ga:15}, where sharp, $k$-explicit $\LtG\rightarrow\LtG$ bounds on $S_k, D_k$ and $D'_k$ were proved using this idea. The present paper extends this method to obtain sharp $\LtG\rightarrow \HoG$ bounds.
The companion paper \cite{GaMuSp:18} then explores the implications of both the $\LtG\rightarrow \LtG$ and $\LtG\rightarrow\HoG$ bounds 
(used in conjunction with the results in Points 3 and 4 above) on the $k$-explicit numerical analysis of the Galerkin method applied to the second-kind equations \eqref{eq:direct} and \eqref{eq:indirect}.

\subsection{Outline of the paper}
In \S\ref{sec:proofL2H1} we prove Theorem \ref{thm:L2H1} (the $\LtG\rightarrow \HoG$ bounds) and Corollary \ref{cor:trace},
and in \S\ref{app:sharpness} we show that the bounds in Theorem \ref{thm:L2H1} are sharp in their $k$-dependence.

\section{Proof of Theorem \ref{thm:L2H1} and Corollary \ref{cor:trace}}\label{sec:proofL2H1}

In this section we prove Theorem \ref{thm:L2H1} and Corollary \ref{cor:trace}. 
The vast majority of the work will be in proving Parts (a) and (b) of Theorem \ref{thm:L2H1}, with Part (c) of Theorem \ref{thm:L2H1} following from the results in \cite[Chapter 4]{Ga:15}, and Corollary \ref{cor:trace} following from the results of \cite{GrLoMeSp:15}.

The outline of this section is as follows:
 In \S\ref{sec:sc} we discuss some preliminaries from the theory of semiclassical pseudodifferential operators, with our default references the texts \cite{EZB} and \cite{ZwScat}.
 In \S\ref{sec:fsh} we recap facts about function spaces on piecewise smooth hypersurfaces.
In \S\ref{sec:rest} we recap restriction bounds on quasimodes -- these results are central to our proof of Theorem \ref{thm:L2H1}.
In \S\ref{sec:proofab} we prove of Parts (a) and (b) of Theorem \ref{thm:L2H1}, 
in \S\ref{sec:proofc} we prove Part (c) of Theorem \ref{thm:L2H1} \S\ref{sec:proofc}, and in \S\ref{sec:cortrace} we prove Corollary \ref{cor:trace}.

We drop the $\lesssim$ notation in this section and state every bound with a constant $C$ (independent of $k$); we do this because later in the proof it will be useful to be able to indicate whether or not the constant in our estimates depends on the order $s$ of the Sobolev space, or on a particular hypersurface $\Gamma$ (we do this via the subscript $s$ and $\Gamma$ -- see, e.g., \eqref{eqn:fourierRestrictionEstimate1} below).

\subsection{Semiclassical Preliminaries}\label{sec:sc}

\subsubsection{Symbols and quantization}
Following \cite[\S3.3]{EZB}, for $k>0$ and $u\in \mathcal{S}(\Rea^d)$, we define the semiclassical Fourier transform ${\mathcal F}_k(u)$ 
by
\beq\label{eq:FT}
{\mathcal F}_k(u)
(\xi):= \int_{\Rea^d} \exp\big(-\ri k \langle y,\xi\rangle\big)u(y)\,\rd y,
\eeq
where $\langle x,\xi\rangle:= \sum_{j=1}^d x_j \xi_j$.
We recall the inversion formula
\beqs
u(x):= \frac{k^d}{(2\pi)^d}\int_{\Rea^d} \exp\big( \ri k \langle x,\xi\rangle\big){\mathcal F}_k(u)(\xi)\,\rd \xi.
\eeqs
We use the standard notation that $D:=-\ri \partial$, so that ${\mathcal F}_k( k^{-1} D_ju)(\xi)=\xi_j{\mathcal F}_k(u)(\xi)$.
We let $\langle\xi\rangle:= (1+|\xi|^2)^{1/2}$ and, following \cite[\S E.1.2]{ZwScat}, we say that $a(\newx, \newxi;k)\in C^\infty(\Rea^{2d}_{x,\newxi})$ lies in $S^m(\Rea^{2d}_{x,\newxi})$ if for all $\alpha,\beta \in \mathbb{N}^d$ and $K\Subset \Rea^d$, there exists $C_{\alpha,\beta, K}>0$ so that 
$$
\sup_{\newx\in K,\newxi\in \Rea^d}\langle \newxi\rangle^{-m+|\beta|}|\partial_{\newx}^\alpha\partial_{\newxi}^\beta a(\newx,\newxi)|\leq C_{\alpha, \beta, K}.
$$
From here on, we follow the usual convention of suppressing  the dependence of $a(x,\newxi;k)$ on $k$, writing instead $a(x,\newxi)$ (see, e.g., \cite[Remark on Page 72]{EZB}), and also writing $S^{m}(\Rea^{2d})$ instead of $S^{m}(\Rea^{2d}_{x,\newxi})$.
We write $S^{-\infty}(\Rea^{2d})=\cap_{m\in \Rea}S^m(\Rea^{2d})$. We say that $a\in S^{\comp}(\Rea^{2d})$ if $a\in S^{-\infty}(\Rea^{2d})$ with $\supp \,a\subset K$ for some compact set $K\subset \Rea^{2d}$ independent of $k$.

For an element $a\in S^m$, we define its quantization to be the operator
\begin{equation}
\label{e:quantize}
u\mapsto a(x,k^{-1}D)u:=\frac{k^{d}}{(2\pi )^d}\int_{\Rea^d}\int_{\Rea^d} \exp\big(\ri k\la x-y,\newxi\ra\big)\,a(x,\newxi)\,u(y)\, \rd y\rd\newxi
\end{equation}
for $u\in \mc{S}(\Rea^d)$. These operators can be defined by duality on $u\in \mc{S}'(\Rea^d)$. We say that an operator $A(k):C_c^\infty(\Rea^d)\to \mc{D}'(\Rea^d)$ is $\O{\Ph{-\infty}{}}(k^{-\infty})$ if it is smoothing (i.e.~its Schwartz kernel ${\mathcal K}$ is smooth) and each seminorm of ${\mathcal K}$ on $C^\infty(\Rea^d\times \Rea^d)$ is $\O{}(k^{-\infty})$. Note that, by introducing an operator $R=\O{\Psi^{-\infty}}(k^{-\infty})$ as an error, we can make the operator $a(x,k^{-1}D)$ properly supported (i.e. so that for any $K\Subset \Rea^{d}$, the kernel $\mathcal{K}$ of $a(x,k^{-1}D) +R$ has the property that both $\pi_{R}^{-1}(K)\cap \supp\, {\mathcal K}$ and $\pi_L^{-1}(K)\cap \supp \,\mathcal{K}$ are compact where $\pi_R,\,\pi_L:\Rea^d\times \Rea^d\to \Rea^d$ are projection onto the right and left factors respectively).

Now, we say that $A(k)$ is a pseudodifferential operator of order $m$ and write $A(k)\in \Ph{m}{}(\Rea^d)$ if $A(k)$ is properly supported and for some $a\in S^m(\Rea^{2d})$,
$$
A(k):=a(x,k^{-1}D)+\O{\Ph{-\infty}{}}(k^{-\infty}).
$$
 We say that $A(k)\in \Ph{\comp}{}(\Rea^d)$ if 
$$A(k)=a(x,k^{-1}D)+\O{\Ph{-\infty}{}}(k^{-\infty})$$ 
for some $a\in S^{\comp}(\Rea^{2d})$. 

Suppose that $A(k)\in \Ph{m}{}(\Rea^d)$ has $A(k)=a(x,k^{-1}D)+\O{\Psi^{-\infty}}(k^{-\infty})$. Then we call $a$ the \emph{full symbol} of $A$. The \emph{principal symbol} of $A\in \Ph{m}{}(\Rea^d)$, denoted by $\sigma(A)$, is defined by 
$$\sigma(A):=a\mod k^{-1}S^{m-1}(\Rea^{2d}).$$

\begin{lemma}{{\bf \cite[Proposition E.16]{ZwScat}}}
\label{lem:comp}
Let $a\in  S^{m_1}(\Rea^{2d})$ and $b\in S^{m_2}(\Rea^{2d})$. 
Then we have
\begin{align*}
a(x,k^{-1}D)b(x,k^{-1}D)&=(ab)(x,k^{-1}D)+k^{-1}r_1(x,k^{-1}D)+\O{\Ph{-\infty}{}}(k^{-\infty})\\
[a(x,k^{-1}D),b(x,k^{-1}D)]&:=a(x,k^{-1}D)b(x,k^{-1}D)-b(x,k^{-1}D)a(x,k^{-1}D)\\
&=\frac{1}{ik}\{a,b\}(x,k^{-1}D)+k^{-2}r_2(x,k^{-1}D)+\O{\Ph{-\infty}{}}(k^{-\infty})
\end{align*}
where $r_1\in S^{m_1+m_2-1}(\Rea^{2d})$, $r_2\in S^{m_1+m_2-2}(\Rea^{2d})$, $\supp \,r_i\subset \supp \,a\cap \supp \,b$, and the Poisson bracket $\{a,b\}$ is defined by 
 $$\{a,b\}:=\sum_{j=1}^d(\partial_{\newxi_j}a)(\partial_{x_j}b)-(\partial_{\newxi_j}b)(\partial_{x_j}a).$$
\end{lemma}

\subsubsection{Action on semiclassical Sobolev spaces}\label{sec:SC}
We define the Semiclassical Sobolev spaces $H_k^s(\Rea^d)$ to be the space $H^s(\Rea^d)$ equipped with the norm
$$\|u\|_{H_k^s(\Rea^d)}^2:=\|\la k^{-1}D\ra^su\|_{L^2(\Rea^d)}^2.$$ 
Note that for $s$ an integer, this norm is equivalent to 
\beqs
\|u\|_{H_k^s(\Rea^d)}^2=\sum_{|\alpha|\leq s}\|(k^{-1}\partial)^\alpha u\|_{L^2(\Rea^d)}^2.
\eeqs
The definition of the semiclassical Sobolev spaces on a smooth compact manifold of dimension $d-1$ $\Gamma$, i.e.~$H^s_k(\Gamma)$ for $|s|\leq 1$,  follows from the definition of $H_k^s(\Rea^{d-1})$  (see, e.g., \cite[Page 98]{Mc:00}). Because solutions of the Helmholtz equation $(-k^{-2}\Delta -1)u=0$ oscillate at frequency $k$, scaling derivatives by $k^{-1}$ makes the $k$-dependence of these norms uniform in the number of derivatives.

With these definitions in hand, we have the following lemma on boundedness of pseudodifferential operators.
\begin{lemma}{{\bf \cite[Proposition E.22]{ZwScat}}}
\label{lem:bound}
Let $A\in \Ph{m}{}(\Rea^d)$. Then for $\chi_1,\chi_2\in C_c^\infty(\Rea^d)$, 
$\|\chi_2A\chi_1\|_{H_k^s(\Rea^d)\to H_k^{s-m}(\Rea^d)}\leq C.$ 
\end{lemma}

\subsubsection{Ellipticity}
For $A\in \Ph{m}{}(\Rea^d)$, we say that $(x,\newxi)\in\Rea^{2d}$ is in the \emph{elliptic set of $A$}, denoted $\Ell(A)$, if there exists $U$ a neighborhood of $(x,\newxi)$ such that for some $\delta>0$,
$$\inf_U|\sigma(A)(x,\newxi)|\geq \delta.$$ 
We then have the following lemma
\begin{lemma}{{\bf \cite[Proposition E.31]{ZwScat}}}
\label{lem:microlocalElliptic}
Suppose that $A\in \Ph{m_1}{}(\Rea^d)$, $b\in S^{\comp}(\Rea^{2d})$ with $\supp\, b\subset \Ell(A)$. Then there exists $R_1,R_2\in \Ph{\comp}{}(\Rea^d)$ with 
$$R_1A=b(x,k^{-1}D)+\O{\Ph{-\infty}{}}(k^{-\infty}),\quad \quad AR_2=b(x,k^{-1}D)+\O{\Ph{-\infty}{}}(k^{-\infty}).$$
Moreover, if $b\in S^{m_2}(\Rea^{2d})$ and there exists $M>0,\,\delta>0$
\beqs
\inf_{ \supp\, b }|\sigma(A)|\la \newxi\ra^{-m_1}>\delta,
\eeqs
 then the same conclusions hold with $R_i\in \Ph{m_2-m_1}{}(\Rea^d)$.
\end{lemma} 

\subsubsection{Pseudodifferential operators on manifolds}
Since we only use the notion of a pseudodifferential operator on a manifold in passing (in Lemma \ref{l:THs} and \S\ref{sec:proofc} below), we simply note that it is possible to define pseudodifferential operators on manifolds (see, e.g., \cite[Chapter 14]{EZB}). The analogues of Lemmas \ref{lem:comp}, \ref{lem:bound}, and \ref{lem:microlocalElliptic} all hold in this setting. Moreover, the principal symbol map can still be defined although its definition is somewhat more involved.

\subsection{Function spaces on piecewise smooth hypersurfaces}\label{sec:fsh}

We now define the spaces $\overline{H^s}(\Gamma)$ and $\dot{H^s}(\Gamma)$ (with the notation for these spaces taken from \cite[\S B.2]{Ho:85}).

\begin{definition}[Extendable Sobolev space $\overline{H^s}(\Gamma)$ on a smooth hypersurface]
\label{d:def-ext-Hs}
Let $\Gamma$ be a smooth hypersurface of $\R^d$ (in the sense of Definition \ref{def:sh}) and let $\Gammaext$ be an extension of $\Gamma$.
Given $s\in\R$, we say that $u \in \overline{H^s}(\Gamma)$  if there exists $\underline{u} \in H^s_{\comp}(\Gammaext)$ such that $\underline{u}|_{\Gamma} = u$.

Let $(U_j , \psi_j)_{j \in J}$ be an atlas of $\Gammaext$ such that $U_j \cap \partial \Gamma \cap \partial \Gammaext = \emptyset$ for all $j \in J$, and let 
\beqs
J_\Gamma:=\big\{j \in J, \,\,U_j \cap \Gamma \neq \emptyset\big\}\quad\tand\quad J_\partial :=\big\{j \in J, \,\,U_j \cap \partial \Gamma \neq \emptyset\big\}
\eeqs
(observe that if $\partial\Gamma=\emptyset$ then $J_{\partial}=\emptyset$).
Let $(\chi_j)_{j \in J}$ be a partition of unity of $\Gammaext$ subordinated to $(U_j)_{j \in J}$.
Given $\chi \in C^\infty_c( {\rm Int} (\Gammaext))$ such that $\chi=1$ in a neighborhood of $\Gamma$, we define 
\begin{equation}
\label{e:def-overHs}
\|u\|_{\overline{H^s}(\Gamma)}= \sum_{j \in J_\Gamma \setminus J_\partial} \|(\chi_ju) \circ \psi_j^{-1}\|_{H^s(\R^{d-1})} 
+\inf_{\underline{u} \in H^s_{\comp}(\Gammaext) , \underline{u}|_\Gamma=u}  \sum_{j \in J_\partial} \| (\chi_j  \chi \underline{u}) \circ \psi_j^{-1}\|_{H^s(\R^{d-1})} .
\end{equation}
\end{definition}
We make two remarks:
\ben
\item The definition of the norm $\overline{H^s}(\Gamma)$ depends on $\Gammaext , \,\chi$, and the choice of charts $(U_j , \psi_j)$ and partition of unity $(\chi_j)$. One can however prove that two different choices of charts $(U_j , \psi_j)$ and partition of unity $(\chi_j)$ lead to equivalent norms $\overline{H^s}(\Gamma)$. In what follows, $(U_j , \psi_j, \chi_j)$ will be traces on $\Gammaext$ of charts and partition of unity on $\R^d$.
\item This definition is the same as, e.g., the definition of $H^s(\Gamma)$ for $\Gamma\subset \Rea^d$ any non-empty open set in \cite[Page 77]{Mc:00}. However, we use the specific notation $\overline{H^s}(\Gamma)$ for the following two reasons: (i) parallelism with the space $\overline{H^s}(\partial \Omega)$ in Definition \ref{def:Hbarbound} below, and (ii) the fact that, without using the overline, 
$H^s(\cdot)$ would be defined differently depending on whether the $\cdot$ is a smooth hypersurface or  the boundary of a Lipschitz domain.
\een

\begin{definition}[Sobolev space $\dot{H}^s(\Gamma)$ on a smooth hypersurface]\label{def:Hdot}
Let $\Gamma$ be a smooth hypersurface of $\R^d$ (in the sense of Definition \ref{def:sh}) and let $\Gammaext$ be an extension of $\Gamma$.
Given $s \in \Rea$, 
We say that $u\in \dot H^s(\Gamma)$ if $u\in H^s_{\comp}(\Gammaext)$ and $\supp\, u\subset \overline{\Gamma}$. Then, 
$$
\|u\|_{\dot{H}^s(\Gamma)}:=\|u\|_{H^s(\Gammaext)}.
$$
\end{definition}

Since $\Gamma$ has $C^0$ boundary, one can show \cite[Theorem 3.3, Lemma 3.15]{SCWSobolev} that the dual of $\overline{H^s}(\Gamma)$ is given by $\dot{H}^{-s}(\Gamma)$ with the dual pairing inherited from that of $H^s_{\comp}(\Gammaext)$ and $H^{-s}_{\comp}(\Gammaext).$ 

For piecewise smooth $\bound$, it is useful to consider the following ``piecewise-$H^s$" spaces.

\begin{definition}[Sobolev space $\overline{H^s}(\bound)$]\label{def:Hbarbound}
Let $\Oi$ be a bounded Lipschitz open set such that its open complement is connected and $\bound$ is a piecewise smooth hypersurface (in the sense of Definition \ref{def:psh}); i.e., $\bound=\cup_{i=1}^N \overline{\Gamma}_i$ where $\Gamma_i$ are smooth hypersurfaces.
With $|s|\leq 1$, we say that $u\in \overline{H^s}(\bound)$ if 
$$
u=\sum_{i=1}^N u_i,\quad \text{ for } \,u_i\in \overline{H^s}(\Gamma_i), \quad\text{ and we let }\quad
\|u\|_{\overline{H^s}(\bound)}:=\sqrt{\sum_{i=1}^N \|u_i\|^2_{\overline{H^s}(\Gamma_i)}}.
$$
\end{definition}

We similarly define the norms $\overline{H^s_k}(\Gamma)$ and $\dot{H}_k^s(\Gamma)$ replacing $\|\cdot \|_{H^s(\R^{d-1})}$ in~\eqref{e:def-overHs} with the weighted-norm $\|\cdot \|_{H_k^s(\R^{d-1})}$ (see, e.g., \cite[Definition E.21]{ZwScat}).

The following lemma implies that, when $S_k$, $D_k$, and $D_k'$ map $\LtG$ to $\HoG$,
to bound the $H^1(\bound)$ norms of $S_k\phi $, $D_k\phi $, and $D_k'\phi$, it is sufficient to bound their $\overline{H^1}(\bound)$ norms.
\ble\label{lem:norm1}
Let $\Oi$ be a bounded Lipschitz open set such that its open complement is connected and $\bound$ is a piecewise smooth hypersurface (in the sense of Definition \ref{def:psh}). If $u\in H^1(\bound)$ then 
\beq\label{eq:norm1}
\|u\|_{H^1(\bound)}\leq \|u\|_{\overline{H^1}(\bound)}
\eeq
\ele
\bpf
Recall that $H^1(\bound)$ can be defined as the completion of $C_{{\rm comp}}^\infty (\bound):=\{ u|_{\bound} : u \in C_0^\infty(\Rea^d)\}$ with respect to the norm
\beq\label{eq:H1}
\int_{\bound}\Big(|\nT u|^2+|u|^2\Big) \rd s
\eeq
\cite[Pages 275-276]{ChGrLaSp:12}
where $\nT$ is the \emph{surface gradient}, defined in terms of a parametrisation of the boundary by, e.g., \cite[Equations (A.13) and (A.14)]{ChGrLaSp:12}.
By the definition of the $\overline{H^1}(\Gamma_i)$ norm from Definition \ref{d:def-ext-Hs}, $u$ restricted to $\Gamma_i$ satisfies 
\begin{align*}
\|u\|_{\overline{H^1}(\Gamma_i)}^2&=
\int_{\Gamma_i}\Big(|\nabla_{\Gamma_i} u|^2+|u|^2\Big) \rd s(\Gamma_i) +\inf_{\underline{u}|_{\Gamma}=u}\int_{{\Gammaext}_i\setminus \Gamma_i} \Big(|\nabla_{{\Gammaext}_i}\underline{u}|^2+|\underline{u}|^2 \Big) \rd s({{\Gammaext}_i)},\\
&\geq \int_{\Gamma_i}\Big(|\nabla_{\Gamma_i} u|^2+|u|^2 \Big)\rd s(\Gamma_i).
\end{align*}
Then, 
\beqs
\|u\|_{H^1(\bound)}^2=\int_{\bound}\Big(|\nT u|^2+|u|^2\Big) \rd s=\sum_{i=1}^N \int_{\Gamma_i}\Big(|\nabla_{\Gamma_i} u|^2+|u|^2 \Big)\rd s(\Gamma_i) \leq \sum_{i=1}^N\|u\|_{\overline{H^1}(\Gamma_i)}^2=\|u\|_{\overline{H^1}(\bound)}^2
\eeqs
and the proof is complete.
\epf

\

Observe that Lemma \ref{lem:norm1} also holds when $H^1(\bound)$ and $\overline{H^1}(\bound)$ are replaced by $H^1_k(\bound)$ and $\overline{H^1_k}(\bound)$ respectively.

\subsection{Recap of restriction estimates for quasimodes }\label{sec:rest}

\begin{theorem}
\label{lem:quasimodeEstimates}
Let $U\subset \Rea^d$ be open and precompact with $\Gamma$ a smooth hypersurface (in the sense of Definition \ref{def:sh})
satisfying $\overline{\Gamma} \subset U$. 
Given $k_0>0$, there exists $C>0$ (independent of $k$) so that if $\|u\|_{L^2(U)}=1$ with 
\beq\label{eq:qm}
(-k^{-2}\Delta -1)u=\O{L^2(U)}(k^{-1}),
\eeq
(i.e.~$\|k^{-2}\Delta u +u\|_{L^2(U)}= \O{} (k^{-1})$)
then, for all $k\geq k_0$,
\begin{equation} \label{restrict}
\|u\|_{L^2(\Gamma)}\leq \begin{cases} Ck^{1/4},\\
Ck^{1/6},&\Gamma\text{ curved,}\end{cases}
\end{equation}
and 
\begin{equation} \label{eqn:restrictEigNormal1} 
\|\partial_\nu u\|_{L^2(\Gamma)}\leq Ck
\end{equation}
where $\partial_\nu$ is a choice of normal derivative to $\Gamma$. 
\end{theorem}

\bpf[References for the Proof of Theorem \ref{lem:quasimodeEstimates}]
The bound \eqref{restrict} for general $\Gamma$ is proved in \cite[Theorem 1.7]{T} and \cite[Theorem 1]{BGT} 
and for curved $\Gamma$ in \cite[Theorem 1.3]{HTacy}.
The bound \eqref{eqn:restrictEigNormal1} is proved in \cite[Theorem 0.2]{T14} (with the analogous estimate for proper eigenfunctions appearing in \cite[Theorem 1.1]{christianson2014exterior}).

We highlight that the analogues of the estimates \eqref{restrict} and \eqref{eqn:restrictEigNormal1} in the context of the wave equation
on smooth Riemannian manifolds appear in \cite[Theorem 1]{Tat} (along with their $L^p$ generalizations in \cite[Theorem 8]{Tat}), with \cite[Pages 187 and 188]{Tat} noting that 
the $L^2$ bounds are a corollary of an estimate in \cite{GreenSeeg}. 
\epf

\bre[Smoothness of $\Gamma$ required for the quasimode estimates]\label{rem:smoothness}
The $k^{1/4}$-bound in \eqref{restrict} is valid when $\Gamma$ is only $C^{1,1}$, and the $k^{1/6}$-bound is valid when $\Gamma$ is $C^{2,1}$ and curved. Therefore, with some extra work it should be possible to prove that the bounds on $S_k$ in Theorem \ref{thm:L2H1} hold with the assumption ``piecewise smooth" replaced by ``piecewise $C^{1,1}$" and ``piecewise $C^{2,1}$ and curved" respectively.
On the other hand, the bound \eqref{eqn:restrictEigNormal1} is not known in the literature for lower regularity $\Gamma$.
\ere

\subsection{Proof of Parts (a) and (b) of Theorem \ref{thm:L2H1}}\label{sec:proofab}

When proving these results, it is more convenient to work in semiclassical Sobolev spaces, i.e.~to prove the bounds from $\LtG$ to $H_k^1(\bound)$,
where (following \S\ref{sec:SC}), 
\beq\label{eq:HohG}
\N{w}_{\HohG}^2:= k^{-2}\N{\nabla_\bound w}_{\LtG}^2+ \N{w}_{\LtG}^2,
\eeq
where $\nT$ is the surface gradient on $\bound$ (defined by, e.g., \cite[Equations (A.13) and (A.14)]{ChGrLaSp:12}).
 We therefore now restate Theorem \ref{thm:L2H1} as Theorem \ref{thm:L2H12} below, working in these spaces.

\begin{theorem}[Restatement of Theorem \ref{thm:L2H1} as bounds from $\LtG\rightarrow H_k^{1}(\bound)$]
\label{thm:L2H12}

\

\noi Let $\Omega$ be a bounded Lipschitz open set such that the open complement $\Oe:= \Rea^d\setminus \overline{\Omega}$ is connected.

\noi (a) 
If $\bound$ is a piecewise smooth hypersurface (in the sense of Definition \ref{def:psh}), then, 
given $k_0>1$, there exists $C>0$ (independent of $k$) such that
\begin{equation}
\label{eqn:optimalFlatSl_new}
\|\Sk\|_{L^2(\bound)\to  H_{k}^1(\bound)}  \leq 
C\,k^{-1/2}\,\log k.
\end{equation}
for all $k\geq k_0$.
Moreover, if $\bound$ is piecewise curved (in the sense of Definition \ref{def:pc}), then, given $k_0>1$, the following stronger estimate holds for all $k\geq k_0$
\begin{equation}\label{eqn:optimalConvexSl_new}
\N{S_k}_{\LtG\rightarrow{H}_k^1(\bound)} \leq C k^{-2/3}\log k.
\end{equation} 

\noi (b) If $\bound$ is a piecewise smooth, $\Dkreg$ hypersurface, for some $\alpha>0$, then, given $k_0>1$, there exists $C>0$ (independent of $k$) such that
\begin{equation}
\label{eqn:optimalFlatDl_new}
\N{D_k}_{\LtG\rightarrow{H}_k^1(\bound)} + \N{D'_k}_{\LtG\rightarrow{H}_k^1(\bound)} \leq C k^{1/4}\log k.
\end{equation}
Moreover, if $\bound$ is piecewise curved, 
then, given $k_0>1$, there exists $C>0$ (independent of $k$) such that the following stronger estimates hold for all $k\geq k_0$
\begin{equation*}
\N{D_k}_{\LtG\rightarrow H^1_k(\bound)} + \N{D'_k}_{\LtG\rightarrow H^1_k(\bound)} \lesssim k^{1/6}\log k.
\end{equation*}

\noi (c) 
If $\Oi$ is convex and $\bound$ is $C^\infty$ and curved (in the sense of Definition \ref{def:curved}) then, given $k_0>1$, there exists $C$ such that, for $k\geq k_0$,
\begin{align*}
&\qquad\qquad\N{S_k}_{\LtG\rightarrow H_k^1(\bound)} \leq C k^{-2/3}, \\
&\N{D_k}_{\LtG\rightarrow H_k^1(\bound)}+\N{D'_k}_{\LtG\rightarrow H_k^1(\bound)} \leq C.
\end{align*}
\end{theorem}

Because Theorem \ref{thm:L2H12} works in the weighted space $H^1_k(\bound)$, the $\LtG\rightarrow \LtG$ bounds contained in Theorem \ref{thm:L2H12} are one power of $k$ stronger than those contained in Theorem \ref{thm:L2H1}. The $\LtG\rightarrow \LtG$ bounds contained in Theorem \ref{thm:L2H12} 
were originally proved in \cite[Theorem 1.2]{GaSm:15}, \cite[Appendix A]{HaTa:15}, and \cite[Theorems 4.29, 4.32]{Ga:15}.  

In \S\ref{sec:outline} below, we give an outline of the proof of Parts (a) and (b). This outline, however, requires the definitions of $S_k$, $D_k$, and $D'_k$ in terms of the free resolvent (a.k.a.~the Newtonian, or volume, potential), given in the next subsection.

\subsubsection{$S_k$, $D_k$, and $D'_k$ written in terms of the free resolvent $R_0(k)$}\label{sec:freeresol}

We now recall the definitions of $S_k$, $D_k$, and $D'_k$ in terms of the free resolvent $R_0(k)$, these expressions are well-known in the theory of BIEs on Lipschitz domains \cite{Co:88}, \cite[Chapters 6 and 7]{Mc:00}. We then specialise these to the case when $\bound$ is a piecewise smooth hypersurface (in the sense of Definition \ref{def:psh})

Let $R_0(k)$ 
be the free (outgoing) resolvent at $k$; i.e.~for $\psi\in C^\infty_{\comp}(\Rea^d)$ we have
\beqs
\big(R_0(k) \psi\big)(x):= \int_{\Rea^d} \Phi_k(x,y) \psi(y) \, \rd \by,
\eeqs
where $\Phi_k(x,y)$ is the (outgoing) fundamental solution defined by \eqref{eq:fund} for $d=2,3$.
Recall that $R_0(k): H^s_{\comp}(\Rea^d)\rightarrow H^{s+2}_{\loc}(\Rea^d)$; see, e.g., \cite[Equation 6.10]{Mc:00}.

With $\Oi$ a bounded Lipschitz open set with boundary $\bound$ and $1/2<s<3/2$, 
let $\gamma^+:H^{s}_{loc}(\Oe) \to H^{s-1/2}(\bound)$ and $\gamma^-:H^{s}(\Omega) \to H^{s-1/2}(\bound)$, be the trace maps \cite[Lemma 3.6]{Co:88}, \cite[Theorem 3.38]{Mc:00}. When $\gamma^+ u = \gamma^- u$ we write both as $\gamma u$ (so that $\gamma:H^{s}_{\loc}(\Rea^d) \to H^{s-1/2}(\bound)$), and we then let $\gamma^*:H^{-s+1/2}(\bound)\to H_{\comp}^{-s}(\Rea^d)$ be the adjoint of $\gamma$ \cite[Equation 6.14]{Mc:00}. Then
$\Sk$ can be written as
\begin{equation} \label{eqn:slo}
\Sk=\gamma R_0(k)\gamma^*
\end{equation}
\cite[Page 202 and Equation 7.5]{Mc:00}, \cite[Proof of Theorem 1]{Co:88}. With $\partial_n^*$ denoting the adjoint of the normal derivative trace (see, e.g., \cite[Equation 6.14]{Mc:00}), we have that the double-layer potential, $\cD_k$, is defined by 
\beqs
\cD_k:= R_0(k) \partial_n^*
\eeqs
\cite[Page 202]{Mc:00}. Recalling that the normal vector $\bn$ points out of $\Oi$ and into $\Oe$, we have that the traces of $\cD_k$ from $\Omega_{\pm}$ to $\Gamma$ are given by 
\beqs
\gamma^{\pm}\cD_k = \pm \half I + D_k 
\eeqs
\cite[Equation 7.5 and Theorem 7.3]{Mc:00} 
and thus 
\beq\label{eq:Dktrue}
D_k = \mp \half I + \gamma^{\pm}R_0(k)\partial_n^*.
\eeq
Similarly, results about the normal-derivative traces of the single-layer potential $\cS_k$ imply that
\beqs
\partial_n^{\pm}\cS_k = \mp \half I + D_k'
\eeqs
so
\beq\label{eq:Dpktrue}
D_k'= \pm \half I + \partial_n^{\pm}R_0(k)\gamma^*.
\eeq

When $\bound$ is Lipschitz, $S_k: L^2(\bound)\rightarrow H^1(\bound)$ by \cite[Theorem 1.6]{Ve:84} (see also, e.g., \cite[Chapter 15, Theorem 5]{MeCo:00},  \cite[Proposition 3.8]{TaylorLipschitzLayer}), and when $\bound$ is $C^{2,\alpha}$ for some $\alpha>0$, then $D_k, D_k': L^2(\bound)\rightarrow H^1(\bound)$ by \cite[Theorem 4.2]{Ki:89} (see also \cite[Theorem 3.6]{CoKr:98}).

We now consider the case when $\bound$ is a piecewise smooth hypersurface (in the sense of Definition \ref{def:psh}) and use the notation that $\widetilde{\Gamma}_i$ are the compact embedded smooth manifolds of $\Rea^d$ such that, for each $i$, $\Gamma_i$ is an open subset of $\widetilde{\Gamma}_i$.
Let $L_i$ be a vector field whose restriction to $\widetilde{\Gamma}_i$ is equal to $\partial_{\nu_{i}}$, the unit normal to $\widetilde{\Gamma}_i$
that is outward pointing with respect to $\bound$.
 Let $\gamma_i:H^{s}_{\loc}(\Rea^d)\to H^{s-1/2}(\Gamma_i)$ denote restriction to $\Gamma_i$. We note that $\gamma_i^*$ is the inclusion map $f\mapsto f\delta_{\Gamma_i}$ where $\delta_{\Gamma_i}$ is $d-1$ dimensional Hausdorff measure on $\Gamma$. Finally, we let $\gamma_i^{\pm}$ denote restrictions from the interior and exterior respectively, where ``interior" and ``exterior" are defined via considering $\Gamma_i$ as a subset of $\bound$.
With these notations, we have that
\beq\label{eqn:dlo}
\Dl=\mp \frac{1}{2}I+\sum_{i,j}\gamma_i^{\pm}R_0(k)L_j^*\gamma_j^*
\eeq
and 
\beq\label{eqn:adlo}
\Dl'=\pm\frac{1}{2}I +\sum_{i,j}\gamma_i^{\pm}L_iR_0(k)\gamma_j^*;
\eeq
the advantage of these last two expressions over \eqref{eq:Dktrue} and \eqref{eq:Dpktrue} is that they involve $\gamma_i$ and $L_i$ instead of $\partial_n^*$ and $\partial_n^{\pm}$.

In the rest of this section, we use the formulae \eqref{eqn:slo}, \eqref{eqn:dlo}, and \eqref{eqn:adlo} as the definitions of $S_k$, $D_k$, and $D'_k$. Note that we slightly abuse notation by omitting the sums in \eqref{eqn:dlo} and \eqref{eqn:adlo} and instead writing 
\beq\label{eq:dlo}
\Dl=\pm \frac{1}{2}I + \gamma^{\pm}R_0(k)L \gamma^*,\quad \Dl'=\mp\frac{1}{2}I+\gamma^{\pm}LR_0(k)\gamma^*.
\eeq

\subsubsection{Outline of the proof of Parts (a) and (b) of Theorem \ref{thm:L2H12}}\label{sec:outline}
The proof of Parts (a) and (b) of Theorem \ref{thm:L2H12} will follow in two steps.
 In Lemma \ref{lem:Q}, we obtain estimates on frequencies $\leq Mk$ and in Lemma \ref{lem:outsideSphere} we complete the proof by estimating the high frequencies ($\geq M k$). 

To estimate the low frequency components, we spectrally decompose the resolvent using the Fourier transform. We are then able to reduce the proof of the low-frequency estimates to the estimates on the restriction of eigenfunctions (or more precisely quasimodes)  to $\pO$ that we recalled in \S\ref{sec:rest}. To understand this reduction, we proceed formally.
From the description of $S_k$ in terms of the free resolvent, \eqref{eqn:slo},
the spectral decomposition of $S_k$ via the Fourier transform is formally
\beq\label{eq:spectral}
S_k f= \int_0^\infty \frac{1}{r^2-(k+\ri 0)^2}\big\la f,\gamma u(r)\big\ra_{L^2(\bound)} \,\gamma u(r)\, \rd r
\eeq
where $u(r)$ is a generalized eigenfunction of $-\Delta$ with eigenvalue $r^2$,
and $k+\ri 0$ denotes the limit of $k+ \ri \varepsilon$ as $\varepsilon \tendo^+$. Observe that the integral in \eqref{eq:spectral} is not well-defined (hence why this calculation is only formal), but \eqref{eq:spectral} nevertheless indicates that estimating $S_k$ amounts to estimating the restriction of the generalized eigenfunction $u(r)$ to $\bound$. 

At very high frequency, we compare the operators $S_k,\,\Dl,$ and $\Dl'$ with the corresponding operators when $k=1$
(recall that the mapping properties of boundary integral operators with $k=1$ have been extensively studied on rough domains; see, e.g. \cite[Chapter 15]{MeCo:00}, \cite{Mc:00}, \cite{TaylorLipschitzLayer}).
By using a description of the resolvent at very high frequency as a pseudodifferential operator, we are able to see that these differences gain additional regularity and hence  obtain estimates on them easily.

The new ingredients in our proof compared to \cite{GaSm:15} and \cite{HaTa:15} are that we have $H^s$ norms in Lemma \ref{lem:Q} and Lemma \ref{lem:outsideSphere} rather than the $L^2$ norms appearing in the previous work.

\subsubsection{Proof of Parts (a) and (b) of Theorem \ref{thm:L2H12}}

\paragraph{Low-frequency estimates.} Following the outline in \S\ref{sec:outline}, our first task is to estimate frequencies $\leq kM$. We start by proving a conditional result that assumes a certain estimate on restriction of the Fourier transform of surface measures to the sphere of radius $r$ (Lemma \ref{lem:Q}).
In Lemma \ref{lem:FourierRestrict} we then show that the hypotheses in Lemma \ref{lem:Q} are a consequence of restriction estimates for quasimodes. In Lemma \ref{lem:LF} we show how the low-frequency estimates on $S_k$, $D_k$, and $D_k'$ follow from Lemma \ref{lem:Q}.

In this section we denote the sphere of radius $r$ by $S_r^{d-1}$ and we denote the surface measure on $S_r^{d-1}$ by $\rd \sigma$. We also use $\widehat{\cdot}$ to denote the \emph{non-semiclassical} Fourier transform, i.e.~$\widehat{u}(\xi)$ is defined by the right-hand side of \eqref{eq:FT} with $k=1$.

\begin{lemma}
\label{lem:Q}
Suppose that for $\Gamma\subset \Rea^d$ any precompact smooth hypersurface, $s\geq 0$, $f\in \dot{H}^{-s}(\Gamma)$, and some $\alpha\,,\,\beta>0$, 
\begin{align}
\label{eqn:fourierRestrictionEstimate1}
\int_{S_r^{d-1}} |\widehat{L^*f \delta_\Gamma}|^2(\newxi)\rd\sigma(\newxi)&\leq C_{\Gamma}\la r\ra^{2\alpha+2s}\|f\|^2_{\dot H^{-s}(\Gamma)},\\
\label{eqn:fourierRestrictionEstimate2}
\int_{S_r^{d-1}} |\widehat{f \delta_\Gamma}|^2(\newxi)\rd\sigma(\newxi)&\leq C_{\Gamma}\la r\ra^{2\beta+2s}\|f\|^2_{\dot H^{-s}(\Gamma)}.
\end{align}
Let $\Gamma_1,\,\Gamma_2\subset \Rea^d$ be compact embedded smooth hypersurfaces.
Recall that $L_i$ is a vector field with $L_i=\partial_{\nu}$ on $\Gamma_i$ for some choice of unit normal $\nu$ on $\Gamma_i$ and $\psi\in \Cc(\Rea)$ with $\psi\equiv 1$ in neighborhood of $0$. 
With the frequency cutoff $\psi(k^{-1}D)$ defined as in \eqref{e:quantize},
we then define for $f\in \dot{H}^{-s_1}(\Gamma_1)$, $g\in \dot{H}^{-s_2}(\Gamma_2)$, $s_i\geq 0$,
\begin{gather*}
\Qs(f,g):=\int_{\Rea^d} R_0(k)(\psi(k^{-1}D)f\delta_{\Gamma_1})\bar{g}\delta_{\Gamma_2} \rd x\,,\quad
\Qa(f,g):=\int_{\Rea^d} R_0(k)(\psi(k^{-1}D)L_1^*(f\delta_{\Gamma_1}))\bar{g}\delta_{\Gamma_2} \rd x,\\
\Qap(f,g):=\int_{\Rea^d} R_0(k)(\psi(k^{-1}D)f\delta_{\Gamma_1})\overline{L_2^*(g\delta_{\Gamma_2})}\, \rd x.
\end{gather*}
Then there exists $C_{\Gamma_1,\Gamma_2,\psi}$ so that for $k>1$,
\begin{align}
\label{eqn:lowFreqSingle}
|\Qs(f,g)|&\leq C_{\Gamma_1,\Gamma_2,\psi}\la k \ra ^{2\beta-1+s_1+s_2}\log \la k\ra\|f\|_{\dot H^{-s_1}(\Gamma_1)}\|g\|_{\dot H^{-s_2}(\Gamma_2)},\\
\label{eqn:lowFreqDouble}
|\Qa(f,g)|+|\Qap(f,g)|&\leq C_{\Gamma_1,\Gamma_2,\psi}\la k \ra ^{\alpha +\beta-1+s_1+s_2}\log \la k\ra\|f\|_{\dot H^{-s_1}(\Gamma_1)}\|g\|_{\dot H^{-s_2}(\Gamma_2)}.
\end{align}
\end{lemma}

\

The key point is that, modulo the frequency cutoff $\psi(k^{-1}D)$, $\Qs(f,g)$, $\Qa(f,g)$, and $\Qap(f,g)$ are given respectively by 
$\la S_k f,g\ra_{\Gamma}, \la D_k f,g\ra_{\Gamma},$ and $\la D_k' f,g\ra_{\Gamma},$ 
where $f$ is supported on $\Gamma_1$ and $g$ on $\Gamma_2$. 

\begin{proof}[Proof of Lemma \ref{lem:Q}]
We follow \cite{GaSm:15}, \cite{HaTa:15} to prove the lemma. First, observe that due to the compact support of $f\delta_{\Gamma_i}$, \eqref{eqn:fourierRestrictionEstimate1} and \eqref{eqn:fourierRestrictionEstimate2}  imply that for $\Gamma\subset \Rea^d$ precompact, 
\begin{align}
\label{eqn:fourierRestrictionEstimate3}
\int_{ S_r^{d-1}}\left|\nabla_{\newxi}\,\widehat{L^* f\delta_\Gamma}(\newxi)\right|^2\rd\sigma(\newxi)&\leq C_\Gamma\,\la r\ra^{2\alpha+2s}\|f\|_{\dot H^{-s}(\Gamma)}^2\,,\\
\label{eqn:fourierRestrictionEstimate4}
\int_{S_r^{d-1}}\left|\nabla_{\newxi}\,\widehat{f\delta_\Gamma}(\newxi)\right|^2\rd\sigma(\newxi)&\leq C_\Gamma\,\la r\ra^{2\beta+2s}\|f\|_{\dot H^{-s}(\Gamma)}^2\,.
\end{align}
Indeed, $\nabla_\xi \widehat{f\delta_\Gamma}=\widehat{xf\delta_{\Gamma}}$
and since $\Gamma$ is compact, 
$$\|xf\|_{\dot H^{-s}(\Gamma)}\leq C_\Gamma\|f\|_{\dot H^{-s}(\Gamma)}.$$
Also, $\nabla_ \xi\widehat{L^*(f\delta_\Gamma)}=\mc{F}(xL^*(f\delta_{\Gamma})).$
Thus,
$$xL^*(f\delta_\Gamma)=L^*(xf\delta_\Gamma)+[x,L^*]f\delta_\Gamma$$
and $[x,L^*]\in C^\infty$. Therefore, using compactness of $\Gamma$,
$$\|xf\|_{\dot H^{-s}(\Gamma)}+\|[x,L^*]f\|_{\dot H^{-s}(\Gamma)}\leq C_\Gamma\|f\|_{\dot H^{-s}(\Gamma)}.$$
Now, $g\delta_{\Gamma_2}\in H^{\min(-s,-1/2-\e)}(\Rea^d)$, $L_2^*(g\delta_{\Gamma_2})\in H^{\min(-s-1,-3/2-\e)}(\Rea^d)$ and since $\psi\in C_c^\infty(\R)$,
\begin{equation*}
\begin{gathered}R_0(k)(\psi(k^{-1}|D|)L^*(f\delta_{\Gamma_1}))\in C^\infty(\Rea^d),\quad \quad R_0(k)(\psi(k^{-1}|D|))f\delta_{\Gamma_1})\in C^\infty(\Rea^d).\end{gathered}
\end{equation*}
By Plancherel's theorem, 
\begin{gather*}
 \Qs(f,g)=\int_{\Rea^d} \psi(k^{-1}|\xi|)\frac{\widehat{f\delta_{\Gamma_1}}(\xi)\overline{\widehat{g\delta_{\Gamma_2}}}(\xi)}{|\xi|^2-(k+\ri 0)^2}\rd \xi,\quad\quad \Qa(f,g)=\int_{\Rea^d} \psi(k^{-1}|\xi|)\frac{\widehat{L_1^* f\delta_{\Gamma_1}}(\xi)\,\overline{\widehat{g\delta_{\Gamma_2}}(\xi)}}{|\xi|^2-(k+\ri 0)^2}\rd \xi,\\
\tand\quad  \Qap(f,g)=\int_{\Rea^d} \psi(k^{-1}|\xi|)\frac{\widehat{ f\delta_{\Gamma_1}}(\xi)\,\overline{\widehat{L_2^*g\delta_{\Gamma_2}}(\xi)}}{|\xi|^2-(k+\ri 0)^2}\rd \xi,
\end{gather*}
where $k+\ri 0$ is understood as the limit of $k+\ri\eps$ as $\eps\to 0^+$. 

Therefore, to prove the lemma, we only need to estimate 
\begin{equation}
\label{eqn:restrictedDualPlancherel}
\int_{\Rea^d} \psi(k^{-1}|\xi|)\frac{F(\xi)\,G(\xi)}{|\xi|^2-(k+\ri 0)^2}\rd \xi
\end{equation}
where, by \eqref{eqn:fourierRestrictionEstimate1}, \eqref{eqn:fourierRestrictionEstimate2}, \eqref{eqn:fourierRestrictionEstimate3}, and \eqref{eqn:fourierRestrictionEstimate4},
\begin{align}\label{eq:ES1} 
&\|F\|_{L^2(S_r^{d-1})}+\|\nabla_{\xi}F\|_{L^2(S_r^{d-1})}\leq C_{\Gamma_1}\la r\ra^{\delta_1+s_1}\|f\|_{\dot H^{-s_1}(\Gamma_1)},\quad\tand\\ \label{eq:ES2}
&\|G\|_{L^2(S_r^{d-1})}+\|\nabla_{\xi}G\|_{L^2(S_r^{d-1})}\leq C_{\Gamma_2}\la r\ra ^{\delta_2+s_2}\|g\|_{\dot H^{-s_2}(\Gamma_2)}.
\end{align}
Consider first the integral in \eqref{eqn:restrictedDualPlancherel} over $\bigl||\xi|-|k|\bigr|\ge 1$. Since $\bigl||\xi|^2-k^2\bigr|\ge \bigl||\xi|^2-|k|^2\bigr|$, by the Schwartz inequality, \eqref{eqn:fourierRestrictionEstimate1}, and \eqref{eqn:fourierRestrictionEstimate2},
this piece of the integral is bounded by
\begin{align}
\int_{\left||\xi|-|k|\right|\geq 1} \left|\psi(k^{-1}|\xi|)\frac{F(\xi)\,G(\xi)}{|\xi|^2-k^2}\right|\rd \xi\!\!\!\!\!\!\!\!\!\!\!\!\!\!\!\!\!\!\!\!\!\!\!\!\!\!\!\!\!\!\!\!\!\!\!\!\!\!\!\!\!\!\!\!\!\!\!\!\!\!\!\!\!\!\!\!\nonumber\\
&\leq \int_{Mk \geq \left|r-|k|\right|\geq 1}\frac{1}{r^2-|k|^2}\int_{S_r^{d-1}}F(r\theta)\,G(r\theta)\rd \sigma(\theta) \rd r\nonumber\\
&\leq C_{\Gamma_1, \Gamma_2}\|f\|_{\dot{H}^{-s_1}(\Gamma_1)}\|g\|_{\dot{H}^{-s_2}(\Gamma_2)}
\int_{M|k|\ge |r-|k||\ge 1}
\la r\ra^{\delta_1+\delta_2+s_1+s_2}\,\bigl|\,r^2-|k|^2\,\bigr|^{-1}\rd r\nonumber\\
&\leq C_{\Gamma_1, \Gamma_2, \psi}\|f\|_{\dot{H}^{-s_1}(\Gamma_1)}\|g\|_{\dot{H}^{-s_2}(\Gamma_2)}|k|^{\delta_1+\delta_2-1+s_1+s_2}\int_{M |k|\geq \left|r-|k|\right|\geq 1}\left|r-|k|\right|^{-1}\rd r\nonumber\\
&\leq C_{\Gamma_1, \Gamma_2,\psi}\,|k|^{\delta_1+\delta_2-1+s_1+s_2}\log |k|\,\|f\|_{\dot H^{-s_1}(\Gamma_1)}\|g\|_{\dot H^{-s_2}(\Gamma_2)},\label{eqn:logLoss}
\end{align}
where the constant $M$ in the intermediate steps depends on the support of $\psi$.
Since $k> 1$, we write
$$
\frac{1}{|\xi|^2-(k+\ri 0)^2}=\frac{1}{|\xi|+(k+\ri 0)}\;\frac{\xi}{|\xi|}\cdot\nabla_\xi\log\big(|\xi|-(k+\ri 0)\big)\,,
$$
where the logarithm is well defined since $\Im(|\xi|-(k+\ri 0))<0$. In particular, we may take the branch cut of the logarithm that has $\log(x)\in \Rea$ for $x\in (0,\infty)$ and has the branch cut on $i[0,\infty)$. Let $\chi(r)=1$ for $|r|\le 1$ and vanish for $|r|\ge 3/2$. We then use integration by parts, together with
\eqref{eq:ES1} and \eqref{eq:ES2}
\begin{multline*}
\Biggl|\;\int_{\Rea^d}\chi(|\xi|-|k|)\,\psi(k^{-1}|\xi|)\,\frac{1}{|\xi|+k+\ri 0}\,
F(\xi)\,G(\xi)\,\;\frac{\xi}{|\xi|}\cdot\nabla_\xi\log\big(|\xi|-(k+\ri 0)\big)\,\rd\xi\;\Biggr|\\
\le C_{\Gamma_1, \Gamma_2, \psi}\,|k|^{\delta_1+\delta_2-1+s_1+s_2}\,\|f\|_{\dot H^{-s_1}(\Gamma_1)}\|g\|_{\dot H^{-s_2}(\Gamma_2)}.
\end{multline*}
Now, taking $\delta_1=\delta_2=\beta$ gives \eqref{eqn:lowFreqSingle}, and taking $\delta_1=\alpha$ and $\delta_2=\beta$ gives \eqref{eqn:lowFreqDouble}.
\end{proof}

\begin{remark} The estimate \eqref{eqn:logLoss} is the only term where the $\log |k|$ appears, which leads to the $\log k$ factors in the bounds of Theorem \ref{thm:L2H1} (without which these bounds would be sharp).
\end{remark}

The proofs of the estimates \eqref{eqn:fourierRestrictionEstimate1} and \eqref{eqn:fourierRestrictionEstimate2} are contained in the following lemma.
\begin{lemma}
\label{lem:FourierRestrict}
Let $\Gamma\subset \Rea^d$ be a precompact smooth hypersurface. Then estimate \eqref{eqn:fourierRestrictionEstimate2} holds with $\beta=1/4$. For $L=\partial_{\nu}$ on $\Gamma$, estimate \eqref{eqn:fourierRestrictionEstimate1} holds with $\alpha=1$. Moreover, if $\Gamma$ is curved then \eqref{eqn:fourierRestrictionEstimate2} holds with $\beta =1/6.$  \end{lemma}

To prove this lemma, we need to understand certain properties of the operator $T_r$ defined by
\begin{equation}
\label{eqn:T}
T_r\phi(x):=\int_{S_r^{d-1}}\phi(\xi)\re^{\ri \la x,\xi\ra }\rd\sigma(\xi).
\end{equation}
Indeed, with $A:H^s(\Rea^d)\to H^{s-1}(\Rea^d)$, to estimate
$$\int_{S_r^{d-1}} |\widehat{A^*(f\delta_{\Gamma})}(\xi)|^2\rd\sigma(\xi),$$ 
we write
\begin{equation}
\label{e:dualize} 
\begin{aligned} 
\la \widehat{A^*(f\delta_{\Gamma})}(\xi),\phi(\xi)\ra_{S_r^{d-1}} &=\int_{S_r^{d-1}}\int_{\Rea^d} A^*(f(x)\delta_{\Gamma})\overline{\phi(\xi)\re^{\ri \la x,\xi\ra}} \rd x\, \rd\sigma(\xi) \\
&=\int_{\Gamma}f \overline{AT_r\phi} \,\rd x=\la f, AT_r\phi\ra_{\Gamma},
\end{aligned}
\end{equation}
with $T_r$ defined by \eqref{eqn:T}.

Before proving Lemma \ref{lem:FourierRestrict} we prove two lemmas (Lemma \ref{l:TboundL2} and \ref{l:THs}) collecting properties of $T_r$.
\begin{lemma}
\label{l:TboundL2}
Let $T_r$ be defined by \eqref{eqn:T} and $\chi \in \Cc(\Rea^d)$. Then,
$$\|\chi T_r\phi\|_{L^2(\Rea^d)}\leq C\|\phi\|_{L^2(S_r^{d-1})}.$$
\end{lemma}
\begin{proof}[Proof of Lemma \ref{l:TboundL2}]
We estimate 
$B:=(\chi T_r)^*\chi T_r:L^2(S_r^{d-1})\to L^2(S_r^{d-1}).$ This operator has kernel
$$B(\xi,\eta)=\int_{\Rea^d} \chi^2(y)\exp{(\ri \la y,\xi-\eta\ra)}\, \rd y=\widehat{\chi^2}(\eta-\xi).$$
Now, for $\eta\in S_r^{d-1}$, and any $N>0$,
$$\int_{S_r^{d-1}}|\widehat{\chi^2}(\eta-\xi)|\,\rd\sigma(\xi)\leq \int_{B(0,r/2)}\la \xi'\ra^{-N}\left[1-\frac{|\xi'|^2}{r^2}\right]^{-1/2}\hspace{-1ex}\rd\xi'+C\la r\ra^{-N}\leq C.$$
Thus, by Schur's inequality, $B$ is bounded on $L^2(S_r^{d-1})$ uniformly in $r$. 
Therefore, 
\beqs
\|\chi T_r\phi\|_{L^2(\Rea^d)}^2 \leq C \|\phi\|_{L^2(S_r^{d-1})}^2.
\eeqs
\end{proof}

In the next lemma, we use $r$ (the radius of $S_r^{d-1}$) as a semiclassical parameter, with the space 
$\overline{H^s_{r}}(\Gamma)$ defined in exactly the same way as $\overline{H^s_k}(\Gamma)$ is defined in \S\ref{sec:fsh}.

\begin{lemma}
\label{l:THs}
With $T_r$ be defined by \eqref{eqn:T}, let $\Gammaext$ denote an extension of $\Gamma$, $\chi \in \Cc(\Rea^d)$ and $A\in \Ph{1}{}(\Rea^d)$ with $\chi \equiv 1$ in a neighborhood of $\Gammaext$. Then for $s\in \Rea $, 
$$\|\chi A T_r\phi\|_{\overline{H^s_{r}}(\Gamma)}\leq C_s\|\chi A T_r\phi\|_{L^2(\Gammaext)}.$$
\end{lemma}
\begin{proof}[Proof of Lemma \ref{l:THs}]
Since $\widehat{T_r\phi}$ is supported on $|\xi|\leq 2r$, $\chi T_r\phi$ is compactly microlocalized in the sense that for $\psi \in \Cc(\Rea)$ with $\psi\equiv 1$ on $[-2,2]$ with support in $[-3,3]$, 
\begin{equation*}
\psi(r^{-1}|D|)\chi A T_r\phi =\chi AT_r\phi +\O{\Ph{-\infty}{}}(r^{-\infty})\chi T_r\phi.
\end{equation*}
(Note that $\psi(r^{-1}|D|)$ can be defined using~\eqref{e:quantize} since $\psi(t)$ is constant near $t=0$.)

Let $\gamma_{\Gammaext}$ denote restriction to $\Gammaext$, and $\gamma|_{\Gamma}$ restriction to $\Gamma$. Let $\chi_\Gamma \in C_c^\infty(\Gammaext)$ with $\chi_\Gamma \equiv 1$ on $\Gamma$. Then for $\psi_1\in \Cc(\Rea)$ with $\psi_1\equiv 1 $ on $[-4,4]$,
$$\chi_\Gamma \psi_1(r^{-1}|D'|_g)\chi_\Gamma \gamma_{\Gammaext} \chi AT_r\phi=\chi_\Gamma^2\gamma_{\Gammaext} \chi AT_r\phi +\O{\Ph{-\infty}{}}(r^{-\infty})\gamma_{\Gammaext} \chi T_r\phi$$
where 
$\psi_1(r^{-1}|D'|_g)$ is a pseudodifferential operator on $\Gammaext$ with symbol $\psi_1(|\xi'|_g)$ and $|\cdot |_g$ denotes the metric induced on $T^*\Gammaext$ from $\Rea^d$ (see Remark \ref{ref:metric} below).

Hence, for $r>1$,
$$\| \gamma_{\Gamma}\chi AT_r\phi\|_{\overline{H^s_r}(\Gamma)}\leq C_s\|\chi A T_r\phi\|_{L^2(\Gammaext)}.$$ 
\end{proof}

\begin{remark}[The definition of $|\cdot|_g$ used in the proof of Lemma \ref{l:THs}]\label{ref:metric}
%While it is standard from Riemannian geometry, we review the definition of $|\cdot|_g$ here. 
We now briefly review the definition of $|\cdot|_g$ from Riemannian geometry.
Observe that the metric on $\Rea^d$ is given by $g_e=\sum_{i=1}^d (\rd y^i)^2$ where $y^i$ are standard coordinates on $\Rea^d$. To induce a metric on $\widetilde{\Gamma}$, at a point $x_0$ we identify $V,W\in T_{x_0}\widetilde{\Gamma}$ with $V',W'\in T_{x_0}\Rea^d$ and define $g(V,W)=g_e(V',W')$. That is, if $V=\sum_i V^i\partial_{y_i}$, $W=\sum_iW^i\partial_{y_i}$, then $g(V,W)=\sum_i V^iW^i$. By doing this at each point $x_0\in \widetilde{\Gamma}$, we obtain a metric on $\widetilde{\Gamma}$, Next, choose coordinates $x^i$ on $\widetilde{\Gamma}$ and write the metric $g$ as $g(\sum a^i\partial_{x^i},\sum b^j\partial_{x^j})=\sum_{ij}g_{ij}(x)a^ib^j$. 
Then, for the corresponding dual coordinates $\xi$ on $T^*\tilde{\Gamma}$, we have $|\xi|_g=\sqrt{\sum_{ij}g^{ij}\xi_i\xi_j}$ where $g^{ij}$ denotes the inverse matrix of $g_{ij}.$ Note that this definition is independent of all of the choices of coordinates.
\end{remark}

We are now in a position to prove Lemma \ref{lem:FourierRestrict}.

\

\bpf[Proof of Lemma \ref{lem:FourierRestrict}]
The key observation for the proof of Lemma \ref{lem:FourierRestrict} is that for $\chi\in \Cc(\Rea^d)$, with $\chi\equiv 1$ in a neighborhood of $\Gamma$,  $\chi T_r\phi$ is a quasimode of the Laplacian with $k=r$ in the sense of \eqref{eq:qm} in Theorem \ref{lem:quasimodeEstimates}. To see this, observe first that 
$-\Delta T_r\phi=r^2T_r\phi$ by the definition \eqref{eqn:T}. Therefore,
\begin{equation*}
-\Delta \chi T_r\phi=r^2\chi T_r\phi+[-\Delta,\chi]T_r\phi.
\end{equation*}
Now, observe that for $\tilde{\chi}\in \Cc(\Rea^d)$ with $\supp\, \tilde{\chi}\subset\{\chi\equiv 1\}$, $\tilde{\chi}[-\Delta,\chi]=0$. Therefore, taking such a $\tilde{\chi}$ with $\tilde{\chi}\equiv 1$ in a neighborhood, $U$ of $\Gamma$ shows that $\chi T_r\phi$ is a quasimode.

To prove \eqref{eqn:fourierRestrictionEstimate2}, we let $A=I$. Then, by the bounds \eqref{restrict} in Theorem \ref{lem:quasimodeEstimates} together with Lemmas \ref{l:TboundL2} and \ref{l:THs}, for $s\geq 0$, 
\begin{equation}
\label{eqn:restrictEig}
\|\chi T_r\phi\|_{\overline{H^{s}}(\Gamma)}\leq C_s \la r\ra^s\|\chi T_r\phi\|_{L^2(\Gammaext)}\leq C_s\la  r\ra^{\frac 14+s}\|\chi T_r\phi\|_{L^2(\Rea^d)}\leq C_s\la r\ra^{\frac14 +s}\|\phi\|_{L^2(S_r^{d-1})},
\end{equation}
and if $\Gamma$ is curved then
\begin{equation}
\label{eqn:restrictEigCurved}
\|\chi T_r\phi\|_{\overline{H^{s}}(\Gamma)}\leq C\la r\ra^{\frac 16+s}\|\phi\|_{L^2(S_r^{d-1})}.
\end{equation}
To prove \eqref{eqn:fourierRestrictionEstimate1}, we take $A=L$. Observe that 
$$\gamma_{\Gammaext} \chi LT_r\phi =\gamma_{\Gammaext} L\chi T_r\phi.$$
Hence, using the fact that $L=\partial_{\nu}$ on $\Gamma$ together with the bound 
\eqref{eqn:restrictEigNormal1} in Theorem \ref{lem:quasimodeEstimates}, we can estimate $LT_r\phi$.
\begin{equation}
\label{eqn:restrictEigNormal}\|\chi LT_r\phi\|_{L^2(\Gammaext)}=  \|L\chi T_r\phi\|_{L^2(\Gammaext)}\leq C\la r\ra\|\chi T_r\phi \|_{L^2(\Rea^d)}.
\end{equation}
In particular, for $s\geq 0$, 
$$\|\chi LT_r\phi\|_{\overline{H^{s}}(\Gamma)}\leq C_s\la r\ra^{s+1}\|\phi\|_{L^2(S_r^{d-1})}.$$ 
Applying the Cauchy-Schwarz inequality together with \eqref{e:dualize}, \eqref{eqn:restrictEig}, \eqref{eqn:restrictEigCurved} and \eqref{eqn:restrictEigNormal} completes the proof of Lemma \ref{lem:FourierRestrict}, since we have shown that 
\begin{align*} 
|\la \widehat{f\delta_\Gamma}(\xi),\phi(\xi)\ra_{L^2(S_r^{d-1})}|&\leq  C_s \la r\ra ^{\frac 14 +s} \|f\|_{\dot H^{-s}(\Gamma)}\|\phi\|_{L^2(S_r^{d-1})},\\
|\la \widehat{L^*(f\delta_\Gamma)}(\xi),\phi(\xi)\ra_{L^2(S_r^{d-1})}|&\leq  C_s \la r\ra^{1 +s} \|f\|_{\dot H^{-s}(\Gamma)}\|\phi\|_{L^2(S_r^{d-1})},
\end{align*}
and if $\Gamma $ is curved,
\begin{align*} 
|\la \widehat{f\delta_\Gamma}(\xi),\phi(\xi)\ra_{L^2(S_r^{d-1})}|&\leq  C_s \la r\ra^{\frac 16 +s} \|f\|_{\dot H^{-s}(\Gamma)}\|\phi\|_{L^2(S_r^{d-1})}.
\end{align*}
\end{proof}

\begin{lemma}[Low-frequency estimates]\label{lem:LF}
Let $s_2$ be either $0$ or $1$.
If $\bound$ is both Lipschitz and piecewise smooth (in the sense of Definition \ref{def:psh}), then
\begin{align}\label{eq:LF1b}
\|\gamma^\pm R_0(k)\psi(k^{-1}D)\gamma^*f\|_{H^{s_2}(\bound)}&\leq C_{\bound, \psi}\la k\ra^{2\beta-1+s_2}\log \la k\ra\|f\|_{L^2(\bound)}\\ \label{eq:LF2b}
\|\gamma R_0(k)\psi(k^{-1}D)L_1^*\gamma^*f\|_{H^{s_2}(\bound)}&\leq C_{\bound, \psi}\la k\ra^{\beta+s_2}\log \la k\ra\|f\|_{ L^2(\bound)}\\ \label{eq:LF3b}
\|\gamma^\pm L_2 R_0(k)\psi(k^{-1}D)\gamma^*f\|_{H^{s_2}(\bound)}&\leq C_{\bound, \psi}\la k\ra^{\beta+s_2}\log \la k\ra\|f\|_{L^2(\bound)}.
\end{align}
with $\beta=1/4$. If $\bound$ is piecewise curved and Lipschitz then \eqref{eq:LF1b}-\eqref{eq:LF3b} hold with $\beta=1/6$.
\ele

\bpf[Proof of Lemma \ref{lem:LF}]
By the duality property of $\overline{H^s}(\Gamma)$ and $\dot{H}^{-s}(\Gamma)$ (discussed after Definition \ref{def:Hdot}), Lemma \ref{lem:FourierRestrict} and the estimates \eqref{eqn:lowFreqSingle} and \eqref{eqn:lowFreqDouble} imply for $s_1,s_2\geq 0$ that there exists $C>0$ independent of $k>1$ so that
\begin{align}\label{eq:LF1a}
\|\gamma_{\Gamma_2}R_0(k)\psi(k^{-1}D)\gamma_{\Gamma_1}^*f\|_{\overline{H^{s_2}}(\Gamma_2)}&\leq C_{\Gamma_1, \Gamma_2, \psi}\la k\ra^{2\beta-1+s_1+s_2}\log \la k\ra\|f\|_{\dot H^{-s_1}(\Gamma_1)},\\ \label{eq:LF2a}
\|\gamma_{\Gamma_2}R_0(k)\psi(k^{-1}D)L_1^*\gamma_{\Gamma_1}^*f\|_{\overline{H^{s_2}}(\Gamma_2)}&\leq C_{\Gamma_1, \Gamma_2, \psi}\la k\ra^{\beta+s_1+s_2}\log \la k\ra\|f\|_{\dot H^{-s_1}(\Gamma_1)},\\ \label{eq:LF3a}
\|\gamma_{\Gamma_2}L_2R_0(k)\psi(k^{-1}D)\gamma_{\Gamma_1}^*f\|_{\overline{H^{s_2}}(\Gamma_2)}&\leq C_{\Gamma_1, \Gamma_2, \psi}\la k\ra^{\beta+s_1+s_2}\log \la k\ra\|f\|_{\dot H^{-s_1}(\Gamma_1)}.
\end{align}
Since $\bound$ is piecewise smooth, $\bound=\sum_{i=1}^N \Gamma_i$ with $\Gamma_i$ smooth hypersurfaces.
Since $\psi(k^{-1}D)$ is a smoothing operator on $\mathcal{S}'$, by elliptic regularity $R_0(k)\psi(k^{-1}D)$ is smoothing and hence its restriction to $\partial\Omega$ maps compactly supported distributions into $H^1(\partial\Omega)$.
Applying \eqref{eq:LF1a}-\eqref{eq:LF3a} with $s_1=0$, $\Gamma=\Gamma_i$, summing over $i$, and using Definition \ref{def:Hbarbound}, we find that, for 
$0\leq s_2\leq 1$,
\begin{align}\label{eq:LF1}
\|\gamma R_0(k)\psi(k^{-1}D)\gamma^*f\|_{\overline{H^{s_2}}(\bound)}&\leq C_{\bound, \psi}\la k\ra^{2\beta-1+s_2}\log \la k\ra\|f\|_{L^2(\bound)}\\ \label{eq:LF2}
\|\gamma^\pm R_0(k)\psi(k^{-1}D)L_1^*\gamma^*f\|_{\overline{H^{s_2}}(\bound)}&\leq C_{\bound, \psi}\la k\ra^{\beta+s_2}\log \la k\ra\|f\|_{ L^2(\bound)}\\ \label{eq:LF3}
\|\gamma^\pm L_2 R_0(k)\psi(k^{-1}D)\gamma^*f\|_{\overline{H^{s_2}}(\bound)}&\leq C_{\bound, \psi}\la k\ra^{\beta+s_2}\log \la k\ra\|f\|_{L^2(\bound)}.
\end{align}
Applying \eqref{eq:LF1}-\eqref{eq:LF3} with $s_2=1$ (using the norm bound \eqref{eq:norm1}) and $s_2=0$, we obtain the estimates \eqref{eq:LF1b}-\eqref{eq:LF3b}.
\epf

\

\paragraph{High frequency estimates.} Next, we obtain an estimate on the high frequency ($\geq kM$) components of $\Sk$, $\Dl$, and $\Dl'$.  We start by analyzing the high frequency components of the free resolvent, proving two lemmata on the structure of the free resolvent there.
\begin{lemma}
\label{lem:freeHighFreq}
Suppose that $z\in [-E,E]$.
Let $\psi\in \Cc(\Rea)$ with $\psi\equiv 1$ on $[-2E^2,2E^2]$. Then for $\chi\in \Cc(\Rea^d).$
$$\chi R_0(zk)\chi(1-\psi(|k^{-1}D|))= B_1,\quad\quad (1-\psi(|k^{-1}D|))\chi R_0(zk)\chi= B_2$$
where $B_i\in k^{-2}\Ph{-2}{}(\Rea^d)$ with 
$$\sigma(B_i)=\frac{\chi^2k^{-2}(1-\psi(|\xi|))}{|\xi|^2-z^2}.$$
\end{lemma}
\begin{proof}[Proof of Lemma \ref{lem:freeHighFreq}]
Let $\chi_0=\chi \in \Cc(\Rea^d)$ and $\chi_n\in \Cc(\Rea^d)$ have $\chi_n\equiv 1$ on $\supp\, \chi_{n-1}$ for $n\geq 1$. Let $\psi_0=\psi\in \Cc(\Rea)$ have $\psi \equiv 1$ on $[-2E^2,2E^2]$, let $\psi_n\in \Cc(\Rea)$ have $\psi_n\equiv 1$ on $[-3E^2/2,3E^2/2]$ and $\supp\, \psi_n\subset \{\psi_{n-1}\equiv 1\}$ for $n\geq 1$. Finally, let $\varphi_n=(1-\psi_n)$ and $\varphi=\varphi_0=(1-\psi)$.  
Then, 
\begin{align}
k^{2}\chi R_0(zk)\chi (-k^{-2}\Delta -z^2)&=(\chi^2-\chi k^{2} \chi_1R_0(zk)\chi_1 [\chi,k^{-2}\Delta]).\label{eqn:ellipticResolve1}
\end{align}
Now, by Lemma \ref{lem:microlocalElliptic} there exists $A_0\in k^{-2}\Ph{-2}{}(\Rea^d)$ with 
\beq\label{eq:Bris1}
A_0=k^{-2}a_0(x,k^{-1}D)+\O{\Ph{-\infty}{}}(k^{-\infty}),\qquad
\supp\, a_0\subset\{\supp\, \varphi_0\}
\eeq
 such that 
\beq\label{eq:Bris2}
k^{2}(-k^{-2}\Delta-z^2)A_0=\varphi(|k^{-1}D|)+\O{\Ph{-\infty}{}}(k^{-\infty})
\eeq
and $A_0$ has 
\beq\label{eq:Bris3}
\sigma(A_0)=\frac{k^{-2}\varphi(|\xi|)}{|\xi|^2-z^2}.
\eeq
(Indeed, since we are working on $\mathbb{R}^d$, 
$$k^{2}(-k^{-2}\Delta-z^2)\frac{k^{-2}\varphi(|k^{-1}D|)}{|k^{-1}D|^2-z^2}=\varphi(|k^{-1}D|)$$
with no remainder.)

Composing \eqref{eqn:ellipticResolve1} on the right with $A_0$, we have 
\begin{align*}
\chi R_0\chi \varphi(|k^{-1}D|)&=\chi^2 A_0-k^{2}\chi \chi_1R_0\chi_1\varphi_1(|k^{-1}D|)[\chi,k^{-2}\Delta]A_0+\O{\Ph{-\infty}{}}(k^{-\infty}),\\
&=\chi^2 A_0+\chi \chi_1R_0\chi_1\varphi_1(|k^{-1}D|)k^{-1}E_1+\O{\Ph{-\infty}{}}(k^{-\infty}),
\end{align*}
where $E_1\in \Ph{-1}{}(\Rea^d)$ and we have used that $\varphi_1\equiv 1$ on $\supp\,\varphi_0$ and hence 
$$
\varphi_1(|k^{-1}D|)[\chi,k^{-2}\Delta]A_0=[\chi,k^{-2}\Delta]A_0+\O{\Ph{-\infty}{}}(k^{-\infty}).
$$
Now, applying the same arguments, but with $A_n$ such that
$$
k^{2}(-k^{-2}\Delta-z^2)A_n=\varphi_n(|k^{-1}D|)E_n+\O{\Ph{-\infty}{}}(k^{-\infty})
$$
there exists $A_n\in k^{-2}\Ph{-2-n}{}(\Rea^d)$ such that
$$\chi_nR_0\chi_n \varphi_n(|k^{-1}D|)E_n=\chi_n^2A_n +{\chi_n}\chi_{n+1}R_0\chi_{n+1}\varphi_{n+1}(|k^{-1}D|)k^{-1}E_{n+1}+\O{\Ph{-\infty}{}}(k^{-\infty})$$
with $E_{n+1}\in \Ph{-1-n}{}(\Rea^d).$
Let
$$
B^N_1:=\chi^2 A_0+\sum_{j=1}^{N-1}\chi k^{-j}A_k.
$$
and assume that 
$$
\chi R_0\chi \varphi(|k^{-1}D|)=B^N_1+\chi \chi_{N}R_0\chi_N \varphi_N(|k^{-d}D|)k^{-N}E_N+\O{\Ph{-\infty}{}}(k^{-\infty}).
$$
for some $N$.
Then,
\begin{align*}
\chi \chi_{N}R_0\chi_N \varphi_N(|k^{-1}D|)k^{-N}E_N&=\chi\big( \chi_N^2A_N +{\chi_N}\chi_{N+1}R_0\chi_{N+1}\varphi_{N+1}(|k^{-1}D|)k^{-1}E_{N+1}+\O{\Ph{-\infty}{}}(k^{-\infty})\big)\\
&=\chi  k^{-N}A_N +\chi \chi_{N+1}R_0\chi_{N+1}\varphi_{N+1}(|k^{-1}D|)k^{-(N+1)}E_{N+1}+\O{\Ph{-\infty}{}}(k^{-\infty})
\end{align*}
and thus by induction, for all $N\geq 1$,
$$
\chi R_0 \chi\varphi(|k^{-1}D|)=B^{N+1}_1+\chi \chi_{N+1}R_0\chi_{N+!}\varphi_{N+1}(|k^{-1}D|)k^{-(N+1)}E_{N+1}+\O{\Ph{-\infty}{}}(k^{-\infty}).
$$
Since $A_k\in \Ph{-2-k}{}(\Rea^d)$, we may let
$$
B_1\sim \chi^2 A_0 +\sum_{j=1}^\infty \chi k^{-j}A_k
$$
to obtain
$$\chi R_0\chi \varphi(|k^{-1}D|)=B_1\in k^{-2}\Ph{-2}{}(\Rea^d)$$
with 
$$\sigma(B_1)=\frac{k^{-2}\chi^2(1-\psi(|\xi|))}{|\xi|^2-z^2}.$$
 The proof of the statement for $B_2$ is identical.
\end{proof}

\

\noi  Next, we prove an estimate on the difference between the resolvent at high energy and that at fixed energy.
\begin{lemma}
\label{lem:highFreqDiff}
Suppose that $z\in [0,E]$. Let $\psi\in \Cc(\Rea)$ with $\psi\equiv 1$ on $[-2E^2,2E^2]$. Then for $\chi \in \Cc(\Rea^d)$, 
$$\chi (R_0(zk)-R_0(1))\chi(1-\psi(|k^{-1}D|))\in k^{-2}\Ph{-4}{}(\Rea^d).$$ 
\end{lemma}
\begin{proof}[Proof of Lemma \ref{lem:highFreqDiff}]
We proceed as in the proof of Lemma \ref{lem:freeHighFreq}.
Let $\chi_0=\chi \in \Cc(\Rea^d)$ and $\chi_n\in \Cc(\Rea^d)$ have $\chi_n\equiv 1$ on $\supp\, \chi_{n-1}$ for $n\geq 1$. Let $\psi_0=\psi\in \Cc(\Rea)$ have $\psi \equiv 1$ on $[-2E^2,2E^2]$, let $\psi_n\in \Cc(\Rea)$ have $\psi_n\equiv 1$ on $[-3E^2/2,3E^2/2]$ and $\supp\, \psi_n\subset \{\psi_{n-1}\equiv 1\}$ for $n\geq 1$. Finally, let $\varphi_n=(1-\psi_n).$  
Then, 
\begin{multline}
\label{eqn:ellipticResolveDiff}
k^{2}\chi (R_0(zk)-R_0(1))\chi (-k^{-2}\Delta -z^2)\\=\chi R_0(1)\left(z^2k^2-1\right)\chi-\chi k^{2} \chi_1(R_0(zk)-R_0(1))\chi_1 [\chi,k^{-2}\Delta]).
\end{multline}
Now, by Lemma \ref{lem:microlocalElliptic} there exists $A_0\in k^{-2}\Ph{-2}{}(\Rea^d)$ such that 
\eqref{eq:Bris1}, \eqref{eq:Bris2}, and \eqref{eq:Bris3} hold.
Composing \eqref{eqn:ellipticResolveDiff} on the right with $k^{-2} A_0$, we have 
\begin{equation}
\label{e:diff}
\begin{aligned}
&\chi (R_0(zk)-R_0(1))\chi \varphi(|k^{-1}D|)\\
&=k^{2}\chi R_0(1)\chi (z^2-k^{-2})A_0-k^{2}\chi \chi_1(R_0(zk)-R_0(1))\chi_1\varphi_1(|k^{-1}D|)[\chi,k^{-2}\Delta]A_0+\O{\Ph{-\infty}{}}(k^{-\infty}).
\end{aligned}
\end{equation}
In particular, iterating using the same argument to write 
\begin{align*}
&\chi_1 (R_0(zk)-R_0(1))\chi _1\varphi_1(|k^{-1}D|)\\
&=k^{2}\chi_1 R_0(1)\chi_1 (z^2-k^{-2})A_1-k^{2}\chi_1 \chi_2(R_0(zk)-R_0(1))\chi_2\varphi_2(|k^{-1}D|)[\chi_1,k^{-2}\Delta]A_1+\O{\Ph{-\infty}{}}(k^{-\infty}),
\end{align*} 
we see that the right hand side of \eqref{e:diff} is in $k^{-2}\Ph{-4}{}(\Rea^d).$
\end{proof}

\

With Lemma \ref{lem:freeHighFreq} and \ref{lem:highFreqDiff} in hand, we obtain the high-frequency estimates of the boundary-integral operators
by comparing them to those at fixed frequency.

\begin{lemma}[High-frequency estimates]
\label{lem:outsideSphere}
Let $M>1$ and $\psi\in \Cc(\Rea)$ with $\psi\equiv 1$ for $|\xi|<M$. Suppose that $\bound$ is both Lipschitz and  
piecewise smooth (in the sense of Definition \ref{def:psh}).
Then for $k>1$ and $\chi\in \Cc(\Rea^d)$
\begin{align}\label{eq:HF1}
\gamma R_0(k)\chi (1-\psi(|k^{-1}D|))\gamma^*&=\O{L^2(\bound)\to {H}_{k}^1(\bound)}(k^{-1}(\log k)^{1/2}).
\end{align}
If, in addition, $\bound$ is $C^{2,\alpha}$ for some $\alpha>0$, then 
\begin{align}\label{eq:HF2}
\mp\frac{1}{2}I+\gamma^{\pm} R_0(k)\chi (1-\psi(|k^{-1}D|))L^*\gamma^*&=\O{L^2(\bound)\to {H}_k^{1}(\bound)}(\log k)\\
\pm\frac{1}{2}I+\gamma^{\pm} LR_0(k)\chi (1-\psi(|k^{-1}D|))\gamma^*&=\O{L^2(\bound)\to {H}_k^{1}(\bound)}(\log k). \label{eq:HF3}
\end{align}
\end{lemma}
\begin{remark}
The factors of $\log k$ in the bounds of Lemma \ref{lem:outsideSphere} are likely artifacts of our proof, but since they do not affect our final results, we do not attempt to remove them here. In fact, if $\bound$ is smooth (rather than piecewise smooth), then one can show that the logarithmic factors can be removed from the bounds in Lemma \ref{lem:outsideSphere} using the analysis in \cite[Section 4.4]{Ga:15}.
\end{remark}
\begin{proof}[Proof of Lemma \ref{lem:outsideSphere}]
By Lemma \ref{lem:highFreqDiff},
\begin{equation*} 
A_k:=\chi (R_0(k)-R_0(1))\chi (1-\psi(k^{-1}D)) \in k^{-2}\Ph{-4}{}.
\end{equation*}
Note that for $s>1/2$,
\begin{equation}\label{eqn:restrictBound}\gamma=\O{H_k^s(\Rea^d)\to H_k^{s-1/2}(\bound)}(k^{1/2});
\end{equation}
this bound follows from repeating the proof of the trace estimate in \cite[Lemma 3.35]{Mc:00} but working in semiclassically rescaled spaces.

Let $B_k:= \gamma A_k\gamma^*$, $C_k:=\gamma A_{k}L^*\gamma^*,$ $C'_k:=\gamma LA_k\gamma^*.$ 
Then, using \eqref{eqn:restrictBound} and the fact that $L,L^*=\O{H_k^s\to H_k^{s-1}}(k)$, we have that $B_k=\O{L^2\to H_k^1}(k^{-1})$ and $C_k,\,C'_k=\O{L^2\to H_k^1}(1)$.

Recalling the notation for $S_k$ \eqref{eqn:slo}, $D_k$, and $D_k'$ \eqref{eq:dlo}, and the mapping properties recapped in \S\ref{sec:freeresol}, we have \begin{align*} 
\gamma   R_0(1)\chi \gamma^*:&\, L^2(\bound)\to {H}^{1}(\bound)
\end{align*}
when $\bound$ is Lipschitz, and
\begin{align*}
\pm\frac{1}{2}I+\gamma^{\pm}  R_0(1)\chi L^*\gamma^*:&\, L^2(\bound)\to {H}^{1}(\bound)\\
\mp\frac{1}{2}I+\gamma^{\pm} L R_0(1)\chi  \gamma^*:&\,L^2(\bound)\to {H}^{1}(\bound)
\end{align*}
when $\bound$ is  $\Dkreg$. 

Now, note that for $\tilde{\Gamma }$ a precompact smooth hypersurface, and $\psi\in \Cc(\Rea)$, 
\begin{equation*}
\begin{gathered}
\|\psi(|k^{-1}D|)\gamma_{\widetilde{\Gamma}}^*\|_{L^2(\widetilde{\Gamma})\to H^s(\Rea^d)}+\|\gamma_{\widetilde{\Gamma}} \psi(|k^{-1}D|)\|_{H^{-s}(\Rea^d)\to \overline{H^{-s-1/2}}(\widetilde{\Gamma})}\leq C \begin{cases} 1&s<-1/2\\
(\log k)^{1/2}&s= -1/2\\
k^{(s+1/2)}&s> -1/2.\end{cases}
\end{gathered}
\end{equation*}
Thus, since $\psi(|k^{-1}D|):\mc{S}'(\R^d)\to C^\infty(\R^d)$ and in particular, $\gamma \psi(|k^{-1}D|):\mc{S}'(\R^d)\to \overline{H^1}(\Gamma)$,
\begin{equation}
\label{e:restrictEsth}
\begin{gathered}
\|\psi(|k^{-1}D|)\gamma^*\|_{L^2(\Gamma)\to H^s(\Rea^d)}+\|\gamma_{\Gamma} \psi(|k^{-1}D|)\|_{H^{-s}(\Rea^d)\to \overline{H^{-s-1/2}}(\Gamma)}\leq C \begin{cases} 1&s<-1/2\\
(\log k)^{1/2}&s= -1/2\\
k^{(s+1/2)}&s> -1/2.\end{cases}
\end{gathered}
\end{equation}
Furthermore, notice that by Lemma \ref{lem:freeHighFreq}, if $\psi_1\in \Cc(\Rea)$ has $\psi_1\equiv 1$ on $\supp\, \psi$, then 
$$\chi R_0(1)\chi \psi(|k^{-1}D|)=\psi_1(|k^{-1}D|)\chi R_0(1)\chi \psi(|k^{-1}D|)+\O{\Ph{-\infty}{}}(k^{-\infty}).$$ 
In particular, using this estimate together with \eqref{e:restrictEsth} and that $\chi R_0(1)\chi:H^s(\Rea^d)\to H^{s+2}(\Rea^d)$, 
\begin{gather*}
\gamma^{\pm}R_0(1)\chi \psi(|k^{-1}D|) \gamma^*=\begin{cases}\O{L^2(\Gamma)\to \overline{H^{1}}(\Gamma)}((\log k)^{1/2})\\
\O{L^2(\Gamma)\to L^2(\Gamma)}(1),
\end{cases}\\
\gamma^{\pm}  R_0(1)\chi\psi(|k^{-1}D|) L^*\gamma^*=\begin{cases}\O{L^2(\Gamma)\to \overline{H^{1}}(\Gamma)}(k)\\
\O{L^2(\Gamma)\to L^2(\Gamma)}(\log k),
\end{cases}\\
\gamma^{\pm} L R_0(1)\chi\psi(|k^{-1}D|)\gamma^*=\begin{cases}\O{L^2(\Gamma)\to \overline{H^{1}}(\Gamma)}(k)\\
\O{L^2(\Gamma)\to L^2(\Gamma)}(\log k),
\end{cases}
\end{gather*} 
Hence,
\begin{align*} 
\gamma  R_0(k)\chi (1-\psi(|k^{-1}D|))\gamma^*&=\gamma  R_0(1)\chi (1-\psi(|k^{-1}D|))\gamma^* +B_k= \O{L^2\to \overline{H^1}}((\log k)^{1/2}).
\end{align*}
Furthermore, since 
$R_0(k)\chi(1-\psi(|k^{-1}D|))\in k^{-2}\Ph{-2}{}(\Rea^d)$, and we have \eqref{eqn:restrictBound}, 
\beq\label{eq:HF1a}
\gamma  R_0(k)\chi (1-\psi(|k^{-1}D|))\gamma^*=\O{L^2\to L^2}(k^{-1}).
\eeq
Next, observe that
\begin{align}\nonumber
\mp\frac{1}{2}+\gamma^{\pm} R_0(k)\chi (1-\psi(|k^{-1}D|))L^*\gamma^*&=\mp\frac{1}{2}+\gamma^{\pm}  R_0(1)\chi (1-\psi(|k^{-1}D|))L^*\gamma^* +C_k\\
&=\begin{cases}\O{L^2\to \overline{H^1}}(k)\\
\O{L^2\to L^2}(\log k),
\end{cases}\label{eq:HF2a}
\end{align} 
\begin{align}\nonumber
\pm\frac{1}{2}+\gamma^{\pm}L R_0(k)\chi (1-\psi(|k^{-1}D|))\gamma^*&=\pm\frac{1}{2}+\gamma^{\pm}  LR_0(1)\chi (1-\psi(|k^{-1}D|))\gamma^* + C_k'\\
&=\begin{cases}\O{L^2\to \overline{H^1}}(k)\\
\O{L^2\to L^2}(\log k).
\end{cases}\label{eq:HF3a}
\end{align} 
Since $\bound$ is piecewise smooth, $\bound = \sum_{i=1}^N \Gamma_i$. 
Applying \eqref{eq:HF1a}-\eqref{eq:HF3a} with $\Gamma=\Gamma_i$, summing over $i$, and then using the result \eqref{eq:norm1} we obtain \eqref{eq:HF1}- \eqref{eq:HF3}
\end{proof}

\

\bpf[Proof of Parts (a) and (b) of Theorem \ref{thm:L2H12}]
This follows from combining the low-frequency estimates \eqref{eq:LF1b}-\eqref{eq:LF3b} in Lemma \ref{lem:LF} with the high-frequency estimates \eqref{eq:HF1}-\eqref{eq:HF3} in Lemma \ref{lem:outsideSphere}, recalling the decompositions \eqref{eqn:slo} and \eqref{eq:dlo}.
\epf

\subsection{Proof of Part (c) of Theorem \ref{thm:L2H12}}\label{sec:proofc}

\begin{proof}[Proof of Part (c) of Theorem \ref{thm:L2H12}]
Observe that \cite[Theorems 4.29, 4.32]{Ga:15} imply that for $\psi\in \Cc(\Rea)$ with $\psi\equiv 1$ on $[-2,2]$, 
$$\psi(|k^{-1}D'|)\Sk=\O{L^2\to H_k^1}(k^{-2/3}),\quad\quad \psi(|k^{-1}D'|)\Dl=\O{L^2\to H_k^1}(1).$$
Then \cite[Lemma 4.25]{Ga:15} shows that $(1-\psi(|k^{-1}D'|))\Sk\in k^{-1}\Ph{-1}{}(\pO)$ and $(1-\psi(|k^{-1}D'|))\Dl\in k^{-1}\Ph{-1}{}(\pO)$ and hence 
$$(1-\psi(|k^{-1}D'|))\Sk=\O{L^2\to H_k^1}(k^{-1}),\quad\quad (1-\psi(|k^{-1}D'|))\Dl=\O{L^2\to H_k^1}(k^{-1}).$$ 
An identical analysis shows that 
$$\Dl'=\O{L^2\to H_k^1}(1).$$ 
\end{proof}

\subsection{Proof of Corollary \ref{cor:trace}}\label{sec:cortrace}

This follows in exactly same way as  \cite[Proof of Corollary 1.2, page 193]{GrLoMeSp:15}. The two ideas are that 
(i) the relationships 
\beqs%\label{eq:dual}
\int_{\Gamma} \phi \,S_k\psi \, \rd s = \int_{\Gamma} \psi \,S_k\phi \, \rd s, \tand 
\int_{\Gamma} \phi \,D_k\psi \, \rd s = \int_{\Gamma} \psi \,D'_k\phi \, \rd s,
\eeqs
for $\phi, \psi \in \LtG$ (see, e.g., \cite[Equation 2.37]{ChGrLaSp:12}),
and the duality of $H^1(\bound)$ and $H^{-1}(\bound)$ (see, e.g., \cite[Page 98]{Mc:00})
allow us to convert bounds on
$S_k$, $D_k$, and $D'_k$ as mappings from $\LtG\rightarrow\HoG$ into bounds on these operators as mappings from $H^{-1}(\bound) \rightarrow \LtG$; and 

\noi (ii) interpolation then allows us to obtain bounds from
$H^{s-1/2}(\bound) \rightarrow H^{s+1/2}(\bound)$ for $|s|\leq 1/2$.

\bre[Using the triangle inequality on $\|D_k'-\ri\eta S_k\|_{\LtG\rightarrow\HoG}$]\label{rem:triangle}
As explained in \S\ref{sec:motivation}, the motivation for proving the $\LtG\rightarrow\HoG$ bounds on $S_k, D_k$, and $D'_k$ is so that they can be used to estimate (in a $k$-explicit way) the smoothing power of $D'_k-\ri \eta S_k$ in the analysis of the Galerkin method via the classic compact-perturbation argument (see \cite[Proof of Theorem 1.10]{GaMuSp:18}). 
We now show that we do not lose anything, from the point of view of $k$-dependence, by using the triangle inequality $\|D_k'-\ri\eta S_k\|_{\LtG\rightarrow\HoG}\leq\|D_k'\|_{\LtG\rightarrow\HoG}+|\eta|\| S_k\|_{\LtG\rightarrow\HoG}$. As a consequence, therefore, the bounds on $\|D_k'-\ri\eta S_k\|_{\LtG\rightarrow\HoG}$ obtained from using the bounds on $\|D_k\|_{\LtG\rightarrow\HoG}$ and $\|S_k\|_{\LtG\rightarrow\HoG}$ in Theorem \ref{thm:L2H1} are sharp.

 First, recall that 
 $D_k'$ and $S_k$ have wavefront set relation given by the billiard ball relation (see for example \cite[Chapter 4]{Ga:15}). Let $B^*\bound$ and $S^*\bound$ denote respectively the unit coball and cosphere bundles in $\bound$. That is, 
 $$
 B^*\bound:=\{(x,\xi)\in T^*\bound : |\xi|_{g(x)}<1\},\qquad S^*\bound:=\{(x,\xi)\in T^*\bound : |\xi|_{g(x)}=1\}.
 $$
  Denote the relation by $C_\beta\subset \overline{B^*\bound}\times\overline{B^*\bound}$ i.e.
 $$ 
 C_\beta=\big\{(x,\xi,y,\eta): (x,\xi)=\beta(y,\eta)\big\}
 $$ 
 where $\beta$ is the billiard ball map (see Figure~\ref{fig:billiardBallmap}).
  To see that the optimal bound in terms of powers of $k$ for 
 $\|D_k'-i\eta S_k\|_{\LtG\rightarrow\HoG}$ is equal to that for $\|D'_k\|_{\LtG\rightarrow\HoG}+|\eta|\|S_k\|_{\LtG\rightarrow\HoG}$, observe that the largest norm for $S_k$ corresponds microlocally to points $(q_1,q_2)\in C_\beta\cap (S^*\bound\times S^*\bound)$ (i.e.~``glancing" to ``glancing"). On the other hand, these points are damped (relative to the worst bounds) for $D_k'$. In particular, microlocally near such points, one expects that 
\begin{equation*}
\N{D_k'f_{q_2}}_{H^1(\bound)}\leq Ck,\qquad \N{S_kf_{q_2}}_{H^1(\bound)}\geq \begin{cases}Ck^{1/2},&\bound \text{ flat},\\Ck^{1/3},&\bound\text{ curved},\end{cases}
\end{equation*}
where $\|f_{q_2}\|_{L^2(\bound)}=1$ and $f_{q_2}$ is microlocalized near $q_2$. 

The norm for $D_k'$ is maximized microlocally near $(p_1,p_2)\in C_\beta\cap (S^*\bound\times B^*\bound)$
 (i.e.~``transversal" to ``glancing"), but near these points, the norm of $S_k$ is damped relative to its worst bound. In particular, microlocally near $(p_1,p_2)$, one expects 
\begin{equation*}
\N{D_k' f_{p_2}}_{H^1(\bound)}\geq \begin{cases}Ck^{5/4},&\bound\text{ flat},\\Ck^{7/6},&\bound\text{ curved},\end{cases}\qquad \N{S_kf_{p_2}}_{H^1(\bound)}\leq \begin{cases}Ck^{1/4},&\bound \text{ flat},\\Ck^{1/6,}&\bound\text{ curved},\end{cases}
\end{equation*}
where $\|f_{p_2}\|_{L^2(\bound)}=1$ and $f_{p_2}$ is microlocalized near $p_2$. 
Therefore, even if $|\eta|$ is chosen so that $\|D'_k\|_{\LtG\rightarrow\HoG}\sim |\eta|\|S_k\|_{\LtG\rightarrow\HoG}$, this analysis shows that there cannot be cancellation since the worst norms occur at different points of phase space.
\ere

 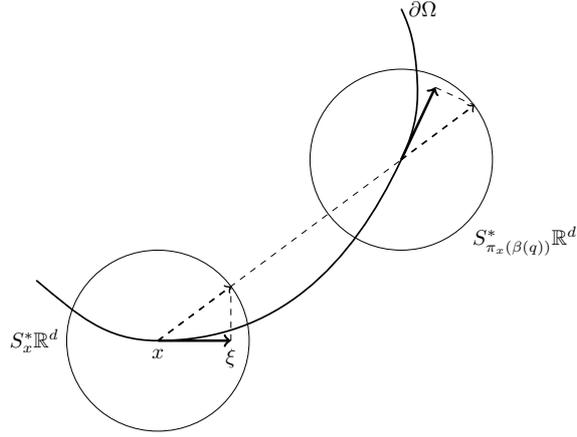
\begin{figure}[h]
\centering
 \scalebox{0.8}{
\begin{tikzpicture}
\begin{scope}[shift={(4,3)}, scale =1.5]
\draw[ultra thin] (0,0) circle [radius=1];
\draw[dashed, thick, ->] (0,0) to (.8,.6);
\draw[very thick,->] (0,0)to (.373,.799);
\draw[thin, dashed] (0.373,0.799)to (0.8,0.6);
\draw (.7,-.9)node[right]{$S_{\pi_x(\beta(q))}^*\R^d$};
\end{scope}
\begin{scope}[scale=1.5]
\draw[ultra thin] (0,0) circle [radius=1];
\draw[very thick,->] (0,0)node[below]{$x$} to (.8,0)node[below]{$\xi$};
\draw[thin, dashed](.8,0) to (.8,.6);
\draw[dashed, thick, ->] (0,0) to (.8,.6);
\draw (-1,0)node[left]{$S_x^*\R^d$};
\end{scope}
\draw [thick](-2,1)to [out=-40, in =180] (0,0) to [out=0, in =-115] (4,3) to [out=65, in =-85] (4.25,4.5) to [out=95, in =-65] (4,5.5);
\draw (4,5.5) node[right]{$\bound$};
\draw[ dashed, thin] (0,0) to (4,3);
\end{tikzpicture}
}
\caption[Billiard ball map]{\label{fig:billiardBallmap}
A recap of the billiard ball map. Let $q=(x,\xi)\in B^*\bound$ (the unit ball in the cotangent bundle of $\bound$). The solid black arrow on the left denotes the covector $\xi\in B_x^*\bound$, with the dashed arrow denoting the unique inward-pointing unit vector whose tangential component is $\xi$. The dashed arrow on the right is the continuation of the dashed arrow on the left, and the solid black arrow on the right is $\xi(\beta(q))\in B_{\pi_x(\beta(q))}^*\bound.$ The center of the left circle is $x$ and that of the right is $\pi_x(\beta(q)).$ If this process is repeated, then the dashed arrow on the right is reflected in the tangent plane at $\pi_x(\beta(q))$: the standard ``angle of incidence equals angle of reflection" rule.
}
\end{figure}

\section{Sharpness of the bounds in Theorem \ref{thm:L2H1}}\label{app:sharpness}

We now prove that the powers of $k$ in the $\|S_k\|_{\LtG\rightarrow\HoG}$ bounds in Theorem \ref{thm:L2H12} are optimal. 
The analysis in \cite[\S A.3]{HaTa:15} proves that the powers of $k$ in the $\|D_k\|_{\LtG\rightarrow\LtG}$ bounds are optimal, but can be adapted in a similar way to below to prove the sharpness of the $\|D_k\|_{\LtG\rightarrow\HoG}$ bounds.

In this section we write $x\in \Rea^d$ as $x= (x', x_d)$ for $x'\in \Rea^{d-1}$, and $x'=(x_1,x'')$ (in the case $d=2$, the $x''$ variable is superfluous).

\ble[Lower bound on $\|S_k\|_{\LtG\rightarrow\HoG}$ when $\bound$ contains a line segment]\label{lem:sharp1}
If $\bound$ contains the set
$$\big\{(x_1,0) : |x_1|<\delta\big\}$$
for some $\delta>0$ and is $C^2$ in a neighborhood thereof
(i.e.~$\bound$ contains a line segment),
then there exists $k_0>0$ and $C>0$ (independent of $k$), such that, for all $k\geq k_0$, 
\beqs
\N{S_k}_{\LtGt} \geq Ck^{-1/2} \quad\tand\quad  \N{S_k}_{\LtG\rightarrow\HoG} \geq Ck^{1/2}.
\eeqs
\ele

Lemma \ref{lem:sharp1} shows that the bound \eqref{eqn:optimalFlatSl}, when $\bound$ is piecewise smooth, is sharp up to a factor of $\log k$.

\ble[General lower bound on $\|S_k\|_{\LtG\rightarrow\HoG}$]\label{lem:sharp2}
If  $\bound$ is $C^2$ in a neighborhood of a point
then there exists $k_0>0$ and $C>0$ (independent of $k$), such that, for all $k\geq k_0$, 
\beqs
\N{S_k}_{\LtGt} \geq Ck^{-2/3} \quad\tand\quad  \N{S_k}_{\LtG\rightarrow\HoG} \geq Ck^{1/3}.
\eeqs
 \ele

Lemma \ref{lem:sharp2} shows that the bound \eqref{eq:k13log}, when $\bound$ is piecewise curved, is sharp up to a factor of $\log k$ and that the bound \eqref{eq:k13}, when $\bound$ is smooth and curved, is sharp.

\bre
The lower bound $\|S_k\|_{\dot{H}^s(\Gamma)\rightarrow H^{s+1}(\Gamma)}\geq Ck^{1/2}$ when $\Gamma$ is a flat screen (i.e.~a bounded and relatively open subset of $\{x \in \Rea^d : x_d=0\}$) and $s\in \Rea$
is proved in \cite[Remark 4.2]{ChHe:15} (recall that $\dot{H}^s(\Gamma)$ is defined in Definition \ref{def:Hdot}).
\ere

\bpf[Proof of Lemma \ref{lem:sharp1}]
By assumption $\Gamma\subset \partial\Omega$, where
$$
\Gamma:=\Big\{ (x_1,x'',\gamma(x')) : |x'|<\delta\Big\}
$$
for some $\gamma(x')$ with 
$\gamma(x_1,0)=0$ for $|x_1|<\delta$ (since the line segment $\{ (x_1,0) : |x_1|<\delta \}\subset \Gamma$).

By the definition of the operator norm, it is sufficient to prove that there exists $u_k\in \LtG$ with $\supp \,u_{k}\subset\Gamma$, $k_0>0$, and $C>0$ (independent of $k$), such that, for all $k\geq k_0$,
 \beq\label{eq:sharp1a}
 \|S_k u_k\|_{L^2(\Gamma)}\geq Ck^{-1/2}\|u\|_{L^2(\Gamma)}\quad\tand\quad \|\partial_{x_1}S_k u_k\|_{L^2(\Gamma)}\geq Ck^{1/2}\|u\|_{L^2(\Gamma)}.
\eeq
\noi We begin by observing that the definition of $\Phi_k(x,y)$ \eqref{eq:fund} and the asymptotics of Hankel functions for large argument and fixed order (see, e.g., \cite[\S10.17]{Di:16}) imply that
\begin{align}
\label{e:PhiForm}\Phi_k(x,y)&=C_dk^{d-2}\re^{\ri k|x-y|}\bigg((k|x-y|)^{-(d-1)/2}+\O{}(k|x-y|)^{-(d+1)/2}\bigg),\\
\la V,\partial_x\ra\Phi_k(x,y)&=C_d'k^{d-1}\frac{\la V,x-y\ra}{|x-y|}\re^{\ri k|x-y|}\bigg((k|x-y|)^{-(d-1)/2}+\O{}(k|x-y|)^{-(d+1)/2}\bigg).
\label{e:PhidForm}
\end{align}
Let $\chi\in \Cc(\Rea)$ with $\supp \chi\subset[-2,2]$, $\chi(0)\equiv1$ on $[-1,1]$ and define 
\beq\label{eq:chi}
\chi_{\e,\gamma_1,\gamma_2}(x')=\chi\big(\e^{-1}k^{\gamma_1}x_1\big)\chi\big(\e^{-1} k^{\gamma_2}|x''|\big),
\eeq
In what follows, we suppress the dependence of $u$ on $k$ for convenience. 
Let
$u(x', \gamma(x')):=\re^{\ri kx_1}\chi_{\e,0,1/2}(x').$
The definition of $\chi$ implies that
\beqs
\supp \,u = \Big\{ 
 (x', \gamma(x'))
: |x_1|\leq 2\e, \, |x''|\leq 2\e k^{-1/2}\Big\},
\eeqs
and thus $\supp \, u\subset \Gamma$ for $\e$ sufficiently small and $k$ sufficiently large (say $\e< (2\sqrt{2})^{-1}\delta$ and $k>1$); for the rest of the proof we assume that $\e$ and $k$ are such that 
this is the case. 
Observe also that 
\beq\label{eq:normu}
\|u\|_{L^2(\Gamma)}\sim C_\e k^{-(d-2)/4}.
\eeq
Let
$$U:=\Big\{ (x', \gamma(x'))
: M\e\leq x_1\leq 2 M\e, \,\,|x''|\leq \e k^{-1/2},\quad M\gg 1\Big\}; $$
the motivation for this choice comes from the analysis in Remark \ref{rem:triangle} below. Indeed, we know that $S_k$ is largest microlocally near points that are glancing in both the incoming and outgoing variables. Since $u$ concentrates microlocally at $x=0$, $\xi=(1,0)$ up to scale $k^{-1/2}$, the billiard trajectory emanating from this point is $\{ t(1,0): t>0\}$. This ray is always glancing since $\Gamma$ is flat. Therefore, we choose $U$ to contain this ray up to scale $k^{-1/2}$. 

Then for $x\in U$, $y\in \supp\, u$, 
\begin{align}\nonumber
|(x',\gamma(x'))-(y',\gamma(y'))|^2&=(x_1-y_1)^2+|x''-y''|^2+|\gamma(x')-\gamma(y')|^2,
\end{align}
Then, observe that by Taylor's formula
$$\gamma(x')-\gamma(y')=\gamma(x_1,0)-\gamma(y_1,0)+\partial_{x''}\gamma(x_1,0)(x''-y'')+y''(\partial_{x''}\gamma(x_1,0)-\partial_{x''}\gamma(y_1,0))+\O{}(|x''|^2+|y''|^2).$$
Since $\gamma(x_1,0)=0$ for $|x_1|<\delta$, 
$$|\gamma(x')-\gamma(y')|^2=\O{}\big(|x''-y''|^2)+\O{}(|x''|^2+|y''|^2).$$
In particular,
\begin{align}\nonumber
|(x',\gamma(x'))-(y',\gamma(y'))|&=(x_1-y_1)+\O{}\Big(\big[|x''-y''|^2+|x''|^2+|y''|^2\big]|x_1-y_1|^{-1}\Big)\\
&=x_1-y_1+\O{}\big(k^{-1}M^{-1}\e\big)\label{eq:exp1},\\
&=x_1\bigg(1+\O{}\big(M^{-1}\big)+\O{}\big(k^{-1}M^{-2}\big)\bigg).\label{eq:exp2}
\end{align}
We have from the Hankel-function asymptotics \eqref{e:PhiForm} and the definition of $u$ that, for $x\in U$,
\begin{align*}
S_ku(x)=C_dk^{d-2}\int_{\Gamma}\re^{\ri k|x-y|+\ri ky_1}&\bigg(k^{-(d-1)/2}|x-y|^{-(d-1)/2}\\
&\qquad+\O{}\left((k|x-y|)^{-(d+1)/2}\right)\bigg)\chi_{\e,0,1.2}(y')\rd s(y),
\end{align*}
and then using the asymptotics \eqref{eq:exp1} in the exponent of the integrand and the asymptotics \eqref{eq:exp2} in the rest of the integrand, we have, for $x\in U$,
\begin{align*}
S_ku(x)=C_dk^{d-2}\frac{\re^{\ri kx_1}}{k^{(d-1)/2}|x_1|^{(d-1)/2}}\int_\Gamma &\big(1+\O{}(M^{-1}\e)\big)\\
&\qquad\bigg(1+\O{}(M^{-1})+\O{\e,M}(k^{-1})\bigg)\chi_{\e,0,1/2}(y')\rd s(y).
\end{align*}
Therefore, with $M$ large enough, $\e$ small enough, and then $k_0$ large enough, the contribution from the integral over $\Gamma$ is determined by the cutoff $\chi_{\e,0,1/2}$, yielding $k^{-(d-2)/2}$, and thus 
\begin{align}\label{eq:final1}
 |S_ku(x')|\geq C k^{(d-2)/2}\frac{1}{k^{(d-1)/2}|x_1|^{(d-1)/2}},\quad x'\in U,\,k\geq k_0.
 \end{align}
In the step of taking $\e$ sufficiently small, we can also take $\e$ small enough to ensure that $U\subset \Gamma$ for all $k\geq 1$.
Using \eqref{eq:final1}, along with the fact that the measure of $U$ $\sim k^{-(d-2)/2}$, we have that
\beq\label{eq:final2}
\|S_ku\|_{L^2(U)}\geq C k^{-1/2-(d-2)/4}.
\eeq
Since we have ensured that $U\subset \Gamma$, \eqref{eq:final2} and \eqref{eq:normu} imply that the first bound in \eqref{eq:sharp1a} holds.
 It easy to see that if we repeat the argument above but with \eqref{e:PhidForm} instead of \eqref{e:PhiForm}, then we obtain the second bound in \eqref{eq:sharp1a}.
\epf 
 
\
 
\bpf[Proof of Lemma \ref{lem:sharp2}] 
Let $x_0\in \bound$ be a point so that $\bound$ is $C^2$ in a neighborhood of $x_0$ and
let $x'$ be coordinates near $x_0$ so that
$$
\Gamma:=\Big\{(x',\gamma(x')) : |x'|<\delta\Big\}\subset\bound,\qquad \text{ with } \gamma\in C^2,\,\gamma(0)=\partial\gamma(0)=0.
$$ 
Similar to the proof of Lemma \ref{lem:sharp1}, it is sufficient to prove that there exists $u_k\in \LtG$ with $\supp \,u_k\subset\Gamma$, $k_0>0$, and $C>0$ (independent of $k$), such that
 \beq\label{eq:sharp2a}
\|S_ku_k\|_{L^2(\Gamma)}\geq Ck^{-1/2}\|u_k\|_{L^2(\Gamma)} \quad\tand\quad 
\|\partial_{x_1}S_k u\|_{L^2(\Gamma)}\geq Ck^{1/2}\|u\|_{L^2(\Gamma)}
\eeq
for all $k\geq k_0$.

The idea in the curved case is the same as in the flat case: choose $u$ concentrating as close as possible to a glancing point and measure near the point given by the billiard map. More practically, this amounts to ensuring that $|x-y|$ looks like $x_1-y_1$ modulo terms that are much smaller than $k^{-1}$. The fact that $\Gamma$ may be curved will force us to choose $u$ differently and cause our estimates to be worse than in the flat case (leading to the weaker -- but still sharp -- lower bound).
 
With $\chi_{\e,\gamma_1,\gamma_2}$ defined by \eqref{eq:chi},
let 
$u(x', \gamma(x')):=\re^{\ri kx_1}\chi_{\e,1/3,2/3}(x')$ where, as in the proof of Lemma \ref{lem:sharp1}, we have $x'=(x_1, x'')$ and 
as in Lemma \ref{lem:sharp1}, $\supp \, u\subset \Gamma$ for $\e$ sufficiently small and $k$ sufficiently large, and for the rest of the proof we assume that this is the case.
Then
\beq\label{eq:normu2}
\|u\|_{L^2(\Gamma)}\leq C_\e k^{-1/6}k^{-(d-2)/3}.
\eeq 
Define
$$ U:=\Big\{ (x', \gamma(x')):M\e k^{-1/3}\leq x_1\leq 2M\e k^{-1/3}, \,\,|x''|\leq \e k^{-2/3},\quad M\gg 1\Big\}. $$
Then, for $y\in \supp \,u$ and $x\in U$, 
\begin{align}\nonumber
|(x',\gamma(x'))-(y',\gamma(x'))|&=(x_1-y_1)+\O{}\big((|x'|^2+|y'|^2)^2|x_1-y_1|^{-1}\big)+\O{}\big(|x''-y''|^2|x_1-y_1|^{-1}\big)\\
&=x_1-y_1+\O{}\big(k^{-1}M^{3}\e^3\big)+\O{}\big(\e k^{-1}M^{-1}\big)\label{eq:exp1a}\\
&=x_1\bigg(1+\O{}\big(M^{-1}\big)+\O{}\big(k^{-2/3}M^{2}\e^2\big)+\O{}\big(k^{-2/3}M^{-2}\big)\bigg).\label{eq:exp2a}
\end{align}
From \eqref{e:PhiForm} and the definition of $u$, we have for $x'\in U$, 
\beqs
S_ku(x)=C_dk^{d-2}\int_{\Gamma}\re^{\ri k|x-y|+\ri k y_1}\bigg(k^{-(d-1)/2}|x-y|^{-(d-1)/2}+\O{}\left((k|x-y|)^{-(d+1)/2}\right)\bigg)\chi_{\e,1/3,2/3}(y')\rd s(y),\\
\eeqs
and then, using \eqref{eq:exp1a} in the exponent of the integrand and \eqref{eq:exp2a} in the rest, we have, for $x'\in U$,
\begin{align*}
S_ku(x)&=C_dk^{d-2}\frac{\re^{\ri kx_1}}{k^{(d-1)/2}|x_1|^{(d-1)/2}}\\
&\qquad \int_\Gamma \Big(1+\O{}(M^3\e^3)+\O{}(M^{-1}\e)\Big)\Big(1+\O{}(M^{-1})+\O{\e,M}(k^{-2/3})\Big)\chi_{\e,1/3,2/3}(y')\rd s(y).
\end{align*}
Thus, fixing $M$ large enough, then $\e$ small enough, then $k_0$ large enough, we have
\begin{align}\label{eq:final1a}
 |S_ku(x')|&\geq C k^{(d-2)/3}\frac{1}{k^{(d-1)/2}|x_1|^{(d-1)/2}}k^{-1/3},\quad x'\in U,\,k\geq k_0.
 \end{align}
In the step of taking $\e$ sufficiently small, we can also take $\e$ small enough so that when $x'\in U$, $|x'|<\delta$, and thus $x'\in \Gamma$.
 Using the lower bound \eqref{eq:final1a}, and the fact that the measure of $U$ $\sim k^{-1/3}k^{-2(d-2)/3}$, we have that
\beqs
\|S_ku\|_{L^2(\Gamma)}\geq \|S_ku\|_{L^2(U)}\geq C k^{-2/3-1/6-(d-2)/3},
 \eeqs
 and so using \eqref{eq:normu2} we obtain the first bound in \eqref{eq:sharp2a}. Similar to before, if we repeat this argument 
with \eqref{e:PhidForm} instead of \eqref{e:PhiForm}, we find the second bound in \eqref{eq:sharp2a}.
\epf

\paragraph{Acknowledgements.}
JG thanks the US National Science Foundation for support under the Mathematical Sciences Postdoctoral Research Fellowship  DMS-1502661.
EAS thanks the UK Engineering and Physical Sciences Research Council for support under Grant EP/R005591/1.

\def\cprime{$'$} \def\cprime{$'$} \def\cprime{$'$}
\footnotesize{

\def\cprime{$'$} \def\cprime{$'$} \def\cprime{$'$}

%\bibliographystyle{plain}
%\bibliography{biblio_acta}
%}

\end{document}